\documentclass[reqno]{amsart}

\usepackage{amsmath}
\usepackage{amssymb}
\usepackage{enumitem}
\usepackage[initials]{amsrefs}
\usepackage{comment}
\usepackage[all]{xy}

\newcount\theTime
\newcount\theHour
\newcount\theMinute
\newcount\theMinuteTens
\newcount\theScratch
\theTime=\number\time
\theHour=\theTime
\divide\theHour by 60
\theScratch=\theHour
\multiply\theScratch by 60
\theMinute=\theTime
\advance\theMinute by -\theScratch
\theMinuteTens=\theMinute
\divide\theMinuteTens by 10
\theScratch=\theMinuteTens
\multiply\theScratch by 10
\advance\theMinute by -\theScratch

\def\today{{\number\day\space
 \ifcase\month\or
  January\or February\or March\or April\or May\or June\or
  July\or August\or September\or October\or November\or December\fi
 \space\number\year}}

\newcommand\barD{{\overline D}}

\newcommand\clspan{{\overline{\mathrm{span}}\,}}
\newcommand\Cpx{{\mathbb C}}
\newcommand\Dc{{\mathcal{D}}} 
\newcommand\DT{\operatorname{DT}}
\newcommand\eps{\varepsilon}
\newcommand\Gbar{{\overline G}}
\newcommand\Hc{{\mathcal{H}}}
\newcommand\Ints{{\mathbb Z}}
\newcommand\lspan{\mathrm{span}\,}
\newcommand\Mcal{{\mathcal{M}}}
\newcommand\Nats{{\mathbb N}}
\newcommand\Reals{{\mathbb R}}

\newcommand\Rt{{\widetilde R}}
\newcommand\supp{\operatorname{supp}}
\newcommand\Tcirc{{\mathbb T}}

\theoremstyle{plain}
\newtheorem{thm}{Theorem}[section]
\newtheorem{lem}[thm]{Lemma}

\newtheorem{prop}[thm]{Proposition}

\newtheorem{thmsubs}{Theorem}[subsection]
\newtheorem{corsubs}[thmsubs]{Corollary} 
\newtheorem{lemsubs}[thmsubs]{Lemma} 
\newtheorem{propsubs}[thmsubs]{Proposition}

\theoremstyle{definition}
\newtheorem{defi}[thm]{Definition}
\newtheorem{ex}[thm]{Example}

\theoremstyle{remark}
\newtheorem{rem}[thm]{Remark}
\newtheorem{claim}{Claim}[thm]

\newtheorem{ques}[thm]{Question}
\newtheorem{ques2}[thm]{Questions}

\usepackage{color}

\begin{document}

\title[Decomposability]{Decomposibility and norm convergence properties in finite von Neumann algebras}

\author[Dykema]{Ken Dykema$^*$}
\address{K. Dykema, Department of Mathematics, Texas A\&M University, College Station, TX, USA.}
\email{ken.dykema@math.tamu.edu}
\author[Noles]{Joseph Noles$^\dag$}
\address{J. Noles, Department of Mathematics, Texas A\&M University, College Station, TX, USA.}
\email{jnoles@math.tamu.edu}
\author[Zanin]{Dmitriy Zanin$^{\S}$}
\address{D. Zanin, School of Mathematics and Statistics, University of new South Wales, Kensington, NSW, Australia.}
\email{d.zanin@math.unsw.edu.au}
\thanks{\footnotesize
{}$^*$This work was supported by a grant from the Simons Foundation/SFARI (524187, K.D.)}
\thanks{\footnotesize
{${}^\dag$Portions of this work are included in the thesis of J.\ Noles for partial fulfillment of the requirements to obtain a Ph.D. degree at Texas A\&M University.}}
\thanks{\footnotesize
${}^{\S}$Research supported by ARC}



\begin{abstract}
We study Schur-type upper triangular forms for elements, $T$, of von Neumann algebras equipped with faithful,
normal, tracial states.
These were introduced in a paper of Dykema, Sukochev and Zanin; they are based on Haagerup--Schultz projections.
We investigate  when the s.o.t.-quasinilpotent part of this decomposition of $T$ is actually quasinilpotent.
We prove implications involving decomposability and strong decomposability of $T$.
We show this is related to norm convergence properties of the sequence $|T^n|^{1/n}$ which, by a result of Haagerup and Schultz,
is known to converge in strong operator topology.
We introduce a Borel decomposability, which is a property appropriate for elements of finite von Neumann algebras,
and show that the circular operator is Borel decomposable.
We also prove the existence of a thin-spectrum s.o.t.-quasinilpotent operator in the hyperfinite II$_1$-factor.
\end{abstract}

\date{October 14, 2017}

\maketitle

\section{Introduction}

The spectrum of an operator on Hilbert space provides important information about it.
Even better is when the operator can be decomposed into pieces using parts of the spectrum.
The Spectral Theorem for normal operators is a strong example of this sort of result.
For non-normal operators, the situation is more complicated.
In the finite dimensional case, we have the Jordan Canonical Form of an operator and the related upper triangular form due to Issai Schur.
In the infinite dimensional case, there are various classes of operators that enjoy nice decomposition properties involving spectrum.
See the book~\cite{LN00} for an excellent treatment.
We will be interested in the {\em decomposable} operators, a notion introduced by Foia\c{s}~\cite{F63} (see~\S\ref{subsec:Decomp} for a definition).

L.\ Brown~\cite{Br86} introduced his spectral distribution measure for an arbitrary operator $T$ in a tracial von Neumann algebra,
by which we mean a von Neumann algebra $\Mcal$ equipped with a normal, faithful, tracial state $\tau$.
This measure, now known as the Brown measure of $T\in\Mcal$, generalizes the spectral counting measure (weighted according to algebraic multiplicity)
for matrices and the usual distribution (i.e., $\tau$ composed with spectral measure) for normal operators.
One issue of interest is that the support of $\nu_T$ is always a subset of the spectrum $\sigma(T)$, but need not be equal to it.

Haagerup and Schultz~\cite{HS09}
proved existence of analogues of generalized eigenspaces for operators $T$ in tracial von Neumann algebras.
Given a Borel subset $B$ of the complex numbers, they found a $T$-hyperinvariant projection $P(T,B)$ satisfying $\tau(P(T,B))=\nu_T(B)$
and splitting the Brown measure according to $B$ and its complement.
For a precise statement see Theorem \ref{thm:hsproj}

In \cite{DSZ} and \cite{Nol}, Sukochev and the authors constructed upper-triangular forms for operators in tracial von Neumann algebras.
These decompositions are of the form $T=N+Q$, where $N$ is normal, $Q$ is s.o.t.-quasinilpotent, and $T$ and $N$ have the same Brown measure.
The constructions generalize the Schur upper triangular form of an $n\times n$ matrix.
The normal part $N$ is constructed as the conditional expectation of $T$ onto an abelian algebra generated by an increasing net of Haagerup-Schultz projections of $T$.
In this paper, we will be concerned only with the upper triangular forms $T=N+Q$ arising from continuous spectral orderings. 
See~\S\ref{subsec:UT} for more details.

Recall, for a bounded operator $A$ on a Hilbert space, the notation $|A| = (A^*A)^{1/2}$ for the positive part of $A$.
A bounded operator $Q$ on Hilbert space is said to be {\em s.o.t.-quasi\-nil\-po\-tent} if $|Q^n|^{1/n}$
converges in strong-operator-topology to $0$ as $n\to\infty$.
A principal motivation for studying these operators is the characterization,
proved by Haagerup and Schultz~\cite{HS09}, that, for elements of a tracial von Neumann algebra,
being s.o.t.-quasinilpotent is equivalent to having Brown measure concentrated at $0$.

The principal results of this paper were motivated by the following question:
\begin{ques}\label{ques:basic}
Given an element $T$ in a tracial von Neumann algebra
and a Schur-type upper-triangular form $T=N+Q$ from~\cite{DSZ},
under what circumstances is the s.o.t.-quasinilpotent
operator $Q$ actually quasinilpotent?
\end{ques}

We give a partial answer to this question in terms of decomposability.
First of all, it should be noted that Haagerup and Schultz proved in~\cite{HS09} that their projection $P(T,B)$ has range space equal to the closure of
the local spectral subspace of $T$ for $B$.
We use this, together with a characterization of decomposability due to Lange and Wang~\cite{LW87} to find a characterizations of decomposability 
for an element $T\in\Mcal$ in terms of spectra of compressions of $T$ by certain projections involving Haagerup--Schultz projections of $T$
--- see Proposition~\ref{prop:disks}.
An operator on Hilbert space is said to be {\em strongly decomposable} if its restrictions to local spectral subspaces for closed subsets of $\Cpx$ are all decomposable.
In Proposition~\ref{prop:strdec}, we characterize strong decomposability in terms of spectra of compressions of $T$.

Our main result, Theorem~\ref{thm:decqn}, includes the implications, for an operator $T\in\Mcal$,
\begin{gather*}
T\text{ is strongly decomposable} \\
\Downarrow \\
\text{in every Schur-type upper triangular form }T=N+Q \\
\text{ from a continuous spectral ordering, }Q\text{ is quasinilpotent} \\
\Downarrow \\
T\text{ is decomposable.}
\end{gather*}

We also relate Question~\ref{ques:basic} to certain norm convergence properties involving positive parts of powers of $T$.
In \cite{HS09}, Haagerup and Schultz show that whenever $T$ is an operator in a tracial von Neumann algebra,
the sequence $|T^n|^{1/n}$ has a strong operator limit as $n\to\infty$,
and that the limit is determined by the Haagerup-Schultz projections of $T$ associated with disks centered at $0$
(see Theorem~\ref{thm:hsconv} below).
This result motivates our next definition, which generalizes the property of being quasinilpotent.
\begin{defi}\label{def:NCP}
An operator $T$
in a C$^*$-algebra
has the {\em norm convergence property} if the sequence $|T^n|^{1/n}$ convergent in norm.
Assuming the C$^*$-algebra is unital,
we say that $T$ has the {\em shifted norm convergence property} if $T-\lambda1$ has the norm convergence property for every complex number $\lambda$.
\end{defi}

A naive guess is that the answer to Question~\ref{ques:basic} is: $Q$ is quasinilpotent if and only if $T$ has the norm convergence property.
However, this is not correct, as we show by explicit construction of a counter-example in Example~\ref{ex}.
A less naive guess is that $Q$ is quasinilpotent if and only if $T$ has the shifted norm convergence property.
This may be true and, in Theorem~\ref{thm:totdisc}, we prove it is true when the Brown measure of $T$ has totally disconnected support.
Furthermore, our main result proves the implication, for general $T\in\Mcal$:
\begin{gather*}
T\text{ is decomposable.} \\
\Downarrow \\
T\text{ has the shifted norm convergence property.}
\end{gather*}

We also, in Section~\ref{sec:Boreldec},
introduce a stronger version of decomposability called {\em Borel decomposability}
(see Definition~\ref{def:fssdecomp})
that is
appropriate to elements of finite von Neumann algebras and natural in connection with Brown measure and Haagerup--Schultz projections.
We show that the DT-operators of~\cite{DH04}, including Voiculescu's circular operator, are all Borel decomposable.

\medskip
We now turn to the topic of the spectrum of an s.o.t.-quasinilpotent operator.
It follows from Remark~4.4 of~\cite{Br86} that, for a general element of a tracial von Neumann algebra,
every connected component of the spectrum must meet the support of the Brown measure;
thus, the spectrum of an s.o.t.-quasinilpotent operator must be a closed, connected set containing $0$.

In the course of these investigations, we also gain knowledge about operators that are s.o.t.-quasinilpotent but not quasinilpotent.

A natural example of an s.o.t.-quasinilpotent operator that is not quasinilpotent is provided by the direct sum
\begin{equation}\label{eq:QJordan}
Q=\oplus_{n=1}^\infty J_n\in\bigoplus_{n=1}^\infty M_n(\Cpx),
\end{equation}
where $J_n$ is the $n\times n$ Jordan block.
Note that this can be realized inside the hyperfinite II$_1$-factor.
Since $Q$ has the same $*$-distribution as $e^{i\theta}Q$ for every real $\theta$
(i.e., its $*$-distribution is invariant under rotation),
and since the spectral radius of $Q$
is easily computed to be $1$, we have that the spectrum of $Q$ is the unit disk centered at the origin.

Also the examples of s.o.t.-quasinilpotent operators found in~\cite{DS09} that are not quasinilpotent
clearly have $*$-distributions that are invariant under rotations
and, thus, have spectra that are disks centered at the origin.
Prior to this writing, every example of such an operator which has appeared in the literature or could be constructed
therefrom using holomorphic functional calculus, has had a spectrum with non-empty interior.
In Theorem~\ref{thm:thinspec},
we construct an s.o.t.-quasinilpotent, non-quasinilpotent operator having thin spectrum (i.e., contained in an interval).
This example is, in turn, used in the aforementioned Example~\ref{ex}.

The contents of the rest of this paper are as follows.
In Section~\ref{sec:background}, we provide some necessary background.
In Section~\ref{sec:Boreldec}, we characterize decomposability and strong decomposability for elements of finite von Neumann algebras in terms
of spectra and Haagerup--Schultz projections.
We also introduce the notion of Borel decomposability and provide some examples.
Section~\ref{sec:mainthm} states and proves our main result and some related results, and asks some related questions.
Section~\ref{sec:thinspec} contains the construction of an s.o.t.-quasinilpotent operator with thin spectrum.
Section~\ref{sec:ncpsotqn} concerns norm convergence properties for s.o.t.-quasinilpotent operators.
Section~\ref{sec:finsupp} focusses on elements having finitely supported Brown measure,
culminating in Theorem~\ref{thm:fsequivs}, giving several equivalent characterizations for such elements.

\section{Preliminaries and notation}
\label{sec:background}

Throughout the paper, the following notation and language will be used: the word trace will refer to a normal, faithful, tracial state.
$\Mcal$ will be a von Neumann algebra of operators on a Hilbert space $\mathcal{H}$ and having a
fixed
trace $\tau$.
Unless otherwise specified, $T$ will be an element of $\Mcal$ and
then $\sigma(T)$ will denote the spectrum of $T$.
Finally, we use the standard notations:
$\Cpx$ is the complex plane, $\mathbb{D}$ is the open unit disk in $\Cpx$ centered at the origin, and $\mathbb{T}$ is the unit circle, namely,
the boundary of $\mathbb{D}$.

\subsection{Brown measure and Haagerup-Schultz projections}

In \cite{Br86}, L.~Brown introduced a generalization of the spectral distribution measure for not necessarily normal operators in tracial von Neumann algebras.

\begin{thmsubs}\label{thm:brnmeas}
Let $T\in\Mcal$.
Then there exists a unique probability measure $\nu_T$ such that for every $\lambda\in\Cpx$, 
$$\int_{[0,\infty)}\log(x)\,d\mu_{|T-\lambda|}(x)=\int_\Cpx \log|z-\lambda|\,d\nu_T(z),$$
where for a positive operator $S$, $\mu_S$ denotes the spectral distribution measure $\tau \circ E$, where $E$ is the spectral measure of $S$.
\end{thmsubs}
The measure $\nu_T$ in Theorem \ref{thm:brnmeas} is called the {\em Brown measure} of $T$.
If $T$ is normal, then $\nu_T$ equals the spectral distribution of $T$.

The following is an easy consequence of an observation made by L.\ Brown~\cite{Br86}.
\begin{propsubs}\label{prop:concomp}
Let $T\in\Mcal$.  Then the support $\supp(\nu_T)$ of the Brown measure of $T$ meets every connected component of the spectrum $\sigma(T)$.
\end{propsubs}
\begin{proof}
By Remark~4.4 of~\cite{Br86}, we have
$\nu_T(C)>0$ for every nonempty, (relatively) clopen subset $C$ of $\sigma(T)$.
Let $K$ be a connected component of $\sigma(T)$.
By a standard argument, there is a sequence $F_1\supseteq F_2\supseteq\cdots$ of clopen subsets of $\sigma(T)$ whose intersection is $K$.
Indeed, for all elements $y\in\sigma(T)\setminus K$, there exists a clopen neighborhood $U_y$ of $y$ in $\sigma(T)$ that is disjoint from $K$.
Since the topological space $\sigma(T)\setminus K$ is Lindel\"of, there is a sequence $y(1),y(2),\ldots$ in $\sigma(T)\setminus K$ such that
$\bigcup_{j=1}^\infty U_{y(j)}=\sigma(T)\setminus K$.
Let $F_k=\sigma(T)\setminus\big(\bigcup_{k=1}^k U_{y(j)}\big)$.

Now, by Brown's observation mentioned above, for each $k$ there is $x_k\in F_k\cap\supp(\nu_T)$.
By compactness of $\sigma(T)$, after replacing $(F_k)_{k=1}^\infty$ by a subsequence, if necessary, we may without loss of generality assume that
the sequence $(x_k)_{k=1}^\infty$ converges to some element $x$ of $\sigma(T)$.
Since each $F_k$ is closed, $x$ belongs to $\bigcap_{k=1}^\infty F_k=K$, and since $\supp(\nu_T)$ is closed, $x\in\supp(\nu_T)$.
\end{proof}

The following is from the main result (Theorem 1.1) of \cite{HS09}.
It provides projections that split the operator $T$ according to the Brown measure.
\begin{thmsubs}\label{thm:hsproj}
For any Borel set $B\subseteq \mathbb{C}$, there exists a unique projection $p=P(T,B)$ such that 
\begin{enumerate}[label=(\roman*),leftmargin=30pt]
\item $Tp = pTp$,
\item $\tau(p) = \nu_T(B)$,
\item when $p\ne0$, considering $Tp$ as an element of $p\mathcal{M}p$, its Brown measure $\nu_{Tp}$ is concentrated in $B$,
\item when $p\ne1$, considering $(1-p)T$ as an element of $(1-p)\mathcal{M} (1-p)$, $\nu_{(1-p)T}$ is concentrated in $\mathbb{C} \setminus B$.
\end{enumerate}
Moreover, $P(T,B)$ is $T$-hyperinvariant and $B_1\subseteq B_2$ implies $P(T,B_1) \leq P(T,B_2)$.
\end{thmsubs}

The projection $P(T,B)$ is called the {\em Haagerup-Schultz projection} of $T$ associated with $B$.

The results about Brown measure and Haagerup-Schultz projections in the following lemma are 
are basic and easy to prove except, perhaps, for the last of them, which is Corollary 7.27 of~\cite{HS09}.
\begin{lemsubs}\label{lem:basicHS}
Let $T\in\Mcal$.
Then for any $\lambda\in\mathbb{C}$ and any Borel set $B\subseteq\mathbb{C}$, letting $B^*$ denote the image of $B$ under complex conjugation, we have 
\begin{enumerate}[label=(\roman*),leftmargin=30pt]
\item $\nu_{(T-\lambda)}(B) = \nu_T(B+\lambda)$\label{i}
\item $\nu_{T^*}(B) = \nu_T(B^*)$\label{ii}
\item $P(T-\lambda,B) = P(T,B+\lambda)$\label{iii}
\item\label{it:PT*} $P(T^*,B) = 1-P(T,\mathbb{C}\setminus B^*)$.\label{iv}
\end{enumerate}
\end{lemsubs}

The following natural lattice properties of Haagerup--Schultz projections were proved in Theorem~3.3 of~\cite{CDSZ}.
\begin{thmsubs}\label{thm:HSprojLattice}
Let $B_1,B_2,\ldots$ be Borel subsets of $\Cpx$.
Then
\begin{align*}
\bigvee_{n=1}^\infty P(T,B_n)&=P\bigg(T,\bigcup_{n=1}^\infty B_n\bigg), \\
\bigwedge_{n=1}^\infty P(T,B_n)&=P\bigg(T,\bigcap_{n=1}^\infty B_n\bigg).
\end{align*}
\end{thmsubs}

Furthermore, the following results were proved as Theorem~3.4 and Corollary~3.5 of~\cite{CDSZ}.
\begin{thmsubs}\label{thm:Pcompress}
Let $T\in\Mcal$.
Suppose $Q$ is a $T$-invariant projection and let $B$ be a Borel subset of $\Cpx$.
Then
\begin{align*}
P^{(Q)}(TQ,B)&=P(T,B)\wedge Q, \\
P^{(1-Q)}((1-Q)T,B)&=Q\vee P(T,B)-Q=\big(Q\vee P(T,B)\big)\wedge(1-Q),
\end{align*}
where $P^{(Q)}(\cdot,\cdot)$ and $P^{(1-Q)}(\cdot,\cdot)$ indicate the Haagerup--Schultz projection computed in the
compression of $\Mcal$ by $Q$ and $1-Q$, respectively.
\end{thmsubs}

In Theorem 8.1 of~\cite{HS09}, Haagerup and Schultz also prove the following theorem.
\begin{thmsubs}\label{thm:hsconv}
Let $T\in\Mcal$.
Then the sequence $|T^n|^{1/n}$ has a strong operator limit $A$, and 
for every $r\ge0$, the spectral projection of $A$ associated with the interval $[0,r]$ is $P(T,r\overline{\mathbb{D}})$.
\end{thmsubs}

If,
for an element $T\in B(\mathcal{H})$,
the sequence $|T^n|^{1/n}$ converges to $0$ in the strong operator topology
as $n\to\infty$, then $T$ is called {\em s.o.t.-quasinilpotent.}
From the above theorem, if follows that, when $T\in\Mcal$,
the property of being s.o.t.-quasinilpotent is independent of the representation of $\Mcal$ on Hilbert space and, moreover,
$T$ is s.o.t.-quasinilpotent if and only if the Brown measure of $T$ is concentrated at $0$.
In other words,
s.o.t.-quasinilpotent operators are spectrally trivial with respect to Brown measure.

\subsection{Upper-triangular forms in tracial von Neumann algebras}
\label{subsec:UT}

The following classical theorem of Schur allows for the upper-triangular forms of square matrices.
\begin{thmsubs}\label{thm:Schur}
Let $A \in M_n(\mathbb{C})$ and let $a_1,a_2,\ldots,a_n$ be the eigenvalues of $A$, listed according to algebraic multiplicity
and in any order.
Then there exists a unitary matrix $U \in M_n(\mathbb{C})$ such that $U^*AU$ is an upper-triangular matrix and $[U^*AU]_{ii}=a_i$,
for every $i\in\{1,\ldots,n\}$.
\end{thmsubs}
We can then decompose the matrix $A$ from Theorem~\ref{thm:Schur} as $A=N+Q$, with $N$ normal, $Q$ nilpotent, and the Brown measure of $N$ identical to the Brown measure of $A$, by letting $\hat{N}$ be the diagonal matrix with diagonal matching that of $U^*AU$, and setting $N=U\hat{N}U^*$.

This decomposition was generalized to tracial von Neumann algebras for continuous spectral orderings in Theorem 6 of~\cite{DSZ}
and later, for more general spectral orderings, in Theorem 3 of \cite{Nol} to give the following theorem.
\begin{thmsubs}\label{thm:utdecomp}
Let $T\in \Mcal$.
Let $\psi:[0,1]\to \mathbb{C}$ be a Borel measurable function such that $\psi ([0,t])$ is Borel for every $t\in [0,1]$, and the set $\{z\in\mathbb{C}\mid \psi^{-1}(z)$ has a minimum$\}$ is a full measure Borel set with respect to $\nu_T$.  Then there exists a spectral measure $E$ satisfying 
\begin{enumerate}[label=(\roman*),leftmargin=30pt]
\item $E(\psi([0,t])) = P(T,\psi([0,t]))$, and hence $T$ and $N=\int_\mathbb{C} z\,dE(z)$ have the same Brown measure, and
\item $Q=T-N$ is s.o.t.-quasinilpotent.
\end{enumerate}
\end{thmsubs}
Many of the results below apply to decompositions of the form $T=N+Q$, where $N=\int_\mathbb{C} zdE(z)$ for a spectral measure $E$ as described in Theorem \ref{thm:utdecomp}.  The term {\em upper-triangular form} will refer to an expression arising in this fashion.
All of our results also require that the function $\psi$ used to generate the spectral measure $E$ be a continuous function; we call this
a {\em continuous spectral ordering}.
Note that for a function $\psi$ to fulfill the requirements for a continuous spectral ordering, it need only be continuous and have range that includes the support of
the Brown measure $\nu_T$.
Thus, the name may be slightly misleading, because the range of $\psi$ need not include the whole spectrum of $T$.

Let us briefly review some facts about the proof of the version of Theorem~\ref{thm:utdecomp} found in~\cite{DSZ}, in the case of
a continuous spectral ordering $\psi$.
Here $\psi$ is any continuous function from $[0,1]$ into the complex plane, whose image contains the support of $\nu_T$.
Letting $\Dc$ be the commutative von Neumann algebra generated by the set of projections
$\{P(T,\psi([0,t]))\mid0\le t\le 1\}$, $N$ is the conditional expectation onto $\Dc$.

Upper-triangular matrices of operators are compatible with spectral theory and Brown measure as described by the following two results.
Though the first is well known, we include a proof for convenience.
\begin{lemsubs}\label{lem:utspec}
Let $T\in\mathcal{M}$, and let $p$ be a $T$-invariant projection with $p \notin \{0,1\}$.
Then $T$ is invertible if and only if $Tp$ and $(1-p)T$ are invertible
in the algebras $p\mathcal{M}p$ and $(1-p)\mathcal{M}(1-p)$, respectively.
It follows that $\sigma(T) = \sigma(Tp)\cup \sigma((1-p)T)$, where the spectra of $Tp$ and $(1-p)T$ are 
for these operators considered as elements of the algebras $p\mathcal{M}p$ and $(1-p)\mathcal{M}(1-p)$, respectively.
\end{lemsubs}
\begin{proof}
If $T$ is invertible, then 
$$(pT^{-1}p)(Tp) = p$$
and
$$((1-p)T)((1-p)T^{-1}(1-p)) = (1-p).$$
Hence $Tp$ has a left inverse, and since $p\mathcal{M}p$ is finite, it follows that $Tp$ is invertible.
Additionally, $(1-p)T$ has a right inverse, and must be invertible.

If $Tp$ and $(1-p)T$ are both invertible, we may write $T$ in the form of a matrix 
$$\left(\begin{matrix}
Tp & pT(1-p)\\
0 & (1-p)T
\end{matrix}\right)$$
which has as an inverse 
$$\left(\begin{matrix}
(Tp)^{-1} & -(Tp)^{-1}pT(1-p)((1-p)T)^{-1} \\
0 & ((1-p)T)^{-1}
\end{matrix}\right).$$

Note that, since $p$ is $T$-invariant, it is also $(T-\lambda)$-invariant for every complex number $\lambda$.
Thus $T-\lambda$ is invertible if and only if both $(T-\lambda)p$ and $(1-p)(T-\lambda)$ are both invertible.
Thus $\sigma(T) = \sigma(Tp)\cup \sigma((1-p)T)$.
\end{proof}

The following result is stated in Proposition 10 of~\cite{DSZ}, and is a consequence of Theorem 2.24 of~\cite{HS07}.
\begin{thmsubs}\label{thm:brmeassplit}
Let $T\in\mathcal{M}$ and let $p$ be a $T$-invariant projection.  Then 
$$\nu_T = \tau(p)\nu_{Tp} + \tau(1-p)\nu_{(1-p)T},$$
where $\nu_S$ denotes the Brown measure of $S$, and $Tp$ and $(1-p)T$ are considered as elements of $p\Mcal p$ and $(1-p)\Mcal (1-p)$, respectively.
\end{thmsubs}

We use Lemma~\ref{lem:utspec} and Theorem \ref{thm:brmeassplit} to give the following corollary.
\begin{corsubs}\label{cor:utquasi}
Let $T\in\Mcal$ and let $p\in\Mcal$ be a $T$-invariant projection.  Then the following two statements hold.
\begin{enumerate}[label=(\roman*),leftmargin=30pt]
\item $T$ is quasinilpotent if and only if $Tp$ and $(1-p)T$ are both quasinilpotent.
\item $T$ is s.o.t.-quasinilpotent if and only if $Tp$ and $(1-p)T$ are both s.o.t.-quasinilpotent.
\end{enumerate}
\end{corsubs}

\subsection{Decomposability and Haagerup--Schultz projections}
\label{subsec:Decomp}

An operator $T$ on a Hilbert space ${\mathcal H}$ is said to be {\em decomposable} if,
for every pair $(U,V)$ of open sets in the complex plane
whose union is the whole complex plane, there are closed, $T$-invariant subspaces ${\mathcal H}'$ and ${\mathcal H}''$
such that the restrictions of $T$
to these have spectra contained in $U$ and $V$, respectively and such that ${\mathcal H}'+{\mathcal H}''={\mathcal H}$.

Given an operator $T$ on a Hilbert space $\mathcal{H}$, a {\it spectral capacity} for $T$ is a mapping $E$ from the collection of closed subsets of $\Cpx$ into the set of all closed $T$-invariant subspaces of $\mathcal{H}$ such that 
\begin{enumerate}[label=(\roman*),leftmargin=30pt]
\item $E(\emptyset)=\{0\}$ and $E(\Cpx)=\mathcal{H}$,
\item $E(\overline{U_1})+E(\overline{U_2})+...+E(\overline{U_n})=\mathcal{H}$ for every finite open cover $\{U_1,U_2,...,U_n\}$ of $\Cpx$,
\item $E\left(\bigcap_{k=1}^\infty F_k\right) = \bigwedge_{k=1}^\infty E(F_k)$ for every countable family $(F_k)_{k=1}^\infty$ of closed subsets of $\Cpx$,
\item letting $P_{E(F)}$ denote the projection with range $E(F)$,
for all closed $F\subseteq\Cpx$, we have $\sigma(TP_{E(F)})\subseteq F$, 
where the spectrum is calculated as an operator on $E(F)$.
\end{enumerate}

Note that, by convention, the operator on the space $\{0\}$ is taken to have empty spectrum.

The following, which is Proposition 1.2.23 of \cite{LN00}, gives conditions equivalent to decomposability for an operator.

\begin{propsubs}\label{decompeqcond}
Let T be a bounded operator on a Hilbert space
H. Then the following are equivalent:
\begin{enumerate}[label=(\roman*),leftmargin=30pt]
\item T is decomposable,
\item T has a spectral capacity,
\item\label{it:dec1-P} for every closed subset $F$ of $\Cpx$, $\mathcal{H}_T(F)$ is closed and 
$$\sigma((1-P_T(F))T)\subseteq \overline{\sigma(T)\setminus F}$$
where $\mathcal{H}_T(F)$ is the local spectral subspace of $T$ corresponding to $F$, $P_T(F)$ is the projection onto $\mathcal{H}_T(F)$,
and the spectrum of $(1-P_T(F))T$ is calculated as an operator on $(1-P_T(F))\mathcal{H}$.
\end{enumerate}
\end{propsubs}
See, for example, \cite{LN00} for more on local spectral theory and decomposability.

Lange and Wang~\cite{LW87} gave an alternative characterization of decomposability that we find useful.
Their result (part of Theorem~2.3 of~\cite{LW87}) in the case of
an operator on a Hilbert space is found immediately below.
Here and elsewhere, by invariant subspace, we always mean closed invariant subspace.
As before, the operator on the space $\{0\}$ is taken to have empty spectrum.

\begin{thmsubs}\label{thm:lw}
Consider an operator $T$ on a Hilbert space $\Hc$.
Then $T$ is decomposable if and only if,
for every pair of open discs $G$,$H$, with $\overline{G}\subseteq H$, there exist $T-$invariant subspaces $Y,Z\subseteq\Hc$ such that 
\begin{enumerate}[label=(\alph*),leftmargin=20pt]
\item $\sigma(TP_Y)\subseteq \Cpx\setminus G$ and $\sigma((1-P_Y)T)\subseteq H$, where $P_Y$ denotes the orthogonal projection of $\Hc$ onto $Y$ and where
the spectra of $TP_Y$ and $(1-P_Y)T$ are computed as operators on $Y$ and $Y^\perp$, respectively,
\item $\sigma(TP_Z)\subseteq H$ and $\sigma((1-P_Z)T)\subseteq \Cpx\setminus G$,
where $P_Z$ denotes the orthogonal projection of $\Hc$ onto $Z$ and where
the spectra of $TP_Z$ and $(1-P_Z)T$ are computed as operators on $Z$ and $Z^\perp$, respectively.
\end{enumerate}
\end{thmsubs}

We will also need the next proposition, which follows from a result of Frunz\u{a}~\cite{Fr76}.
\begin{propsubs}\label{prop:T*decomp}
If an operator $T$ on a Hilbert space is decomposable, then its adjoint $T^*$ is decomposable.
\end{propsubs}

Haagerup and Schultz proved a relation between the hyperinvariant subspaces that they constructed and local spectral subspaces
for decomposable operators in finite von Neumann algebras.
The following is Proposition~9.2 of~\cite{HS09}:

\begin{propsubs}\label{prop:HSdecomp}
Suppose $T\in\Mcal$ is a decomposable operator.
Then for every Borel set $X\subseteq\Cpx$,
the Haagerup-Schultz projection $P(T,X)$
equals the projection onto the closure of the local spectral subspace ${\mathcal H}_T(X)$ of $T$ for $X$.
\end{propsubs}

Note that (see Proposition 1.2.19 of~\cite{LN00}), if $F$ is a closed subset of $\Cpx$ and if $T$ is decomposable,
then the local spectral subspace $\Hc_T(F)$ is closed.

\section{Decomposability, Borel decomposability and Haagerup--Schultz projections}
\label{sec:Boreldec}

Lange and Wang's criterion (Theorem~\ref{thm:lw}) for decomposability,
together with Proposition~\ref{prop:HSdecomp} of Haagerup and Schultz, yield the following characterization.
\begin{prop}\label{prop:disks}
Let $T\in\Mcal$.
Then the following are equivalent:
\begin{enumerate}[label=(\roman*),leftmargin=25pt]
\item\label{it:Tdec} $T$ is decomposable.
\item\label{it:sigmaF} For all closed subsets $F$ of $\Cpx$,
\begin{align}
\sigma(TP(T,F))&\subseteq F, \label{eq:sigF} \\
\sigma((1-P(T,\Cpx\setminus F))T)&\subseteq F, \label{eq:sig1-F}
\end{align}
where the spectra are computed in the compressions of $\Mcal$ by the projections  $P(T,F)$ and $1-P(T,\Cpx\setminus F)$, respectively.
\item\label{it:sigmaA} For all Borel subsets $A$ of $\Cpx$,
\begin{align}
\sigma(TP(T,A))&\subseteq\overline{A}, \label{eq:sigA} \\
\sigma((1-P(T,\Cpx\setminus A))T)&\subseteq\overline{A}, \label{eq:sig1-A}
\end{align}
where the spectra are computed in the compressions of $\Mcal$ by the projections  $P(T,A)$ and $1-P(T,\Cpx\setminus A)$, respectively.
\item\label{it:sigmaG}  For every open disk $D$ in $\Cpx$,
\begin{align}
\sigma(TP(T,\barD))&\subseteq\barD, \label{eq:sigG} \\
\sigma(TP(T,\Cpx\setminus D))&\subseteq\Cpx\setminus D, \label{eq:sigGc} \\
\sigma((1-P(T,\barD))T)&\subseteq \Cpx\setminus D, \label{eq:sig1-G} \\
\sigma((1-P(T,\Cpx\setminus D))T)&\subseteq\barD, \label{eq:sig1-Gc}
\end{align}
where the spectra are computed in the compressions of $\Mcal$ by the projections $P(T,\barD)$, $P(T,\Cpx\setminus D)$,
$1-P(T,\barD)$ and $1-P(T,\Cpx\setminus D)$, respectively,
\item\label{it:sigmaGbdy}  For every open disk $D$ in $\Cpx$ such that $\nu_T(\partial D)=0$, the inclusions \eqref{eq:sigG}--\eqref{eq:sig1-Gc} hold.
\end{enumerate}
\end{prop}
\begin{proof}
We first show \ref{it:Tdec}$\implies$\ref{it:sigmaF}.
Suppose $T$ is decomposable.
By Proposition~\ref{prop:HSdecomp} and the properties of spectral subspaces, the inclusion~\eqref{eq:sigF} follows.
But then, since (see Proposition~\ref{prop:T*decomp}) $T^*$ is also decomposable and by using Lemma~\ref{lem:basicHS}\ref{it:PT*}, we have
\[
\sigma((1-P(T,\Cpx\setminus F))T)=\sigma(P(T^*,F^*)T)=\sigma(T^*P(T^*,F^*))^*\subseteq (F^*)^*=F,
\]
where we have used the notation $X^*$ for the complex conjugate of a subset of $\Cpx$.
This proves~\eqref{eq:sig1-F}.

We now show \ref{it:sigmaF}$\implies$\ref{it:sigmaA}.
Since $P(T,A)$ is a $TP(T,\overline{A})$-invariant subprojection of $P(T,\overline{A})$,
by Lemma~\ref{lem:utspec}
we have $\sigma(TP(T,A))\subseteq\sigma(TP(T,\overline{A}))$, where the spectra are computed in the compressions
of $\Mcal$ by $P(T,A)$ and $P(T,\overline{A})$, respectively.
Thus, invoking~\ref{it:sigmaF} yields~\eqref{eq:sigA}.
Moreover, $1-P(T,\Cpx\setminus A)$ is a $(1-P(T,\Cpx\setminus\overline{A}))T$-coinvariant subprojection of $1-P(T,\Cpx\setminus\overline{A})$,
so by Lemma~\ref{lem:utspec}
we have
\[
\sigma\big((1-P(T,\Cpx\setminus A))T\big)\subseteq\sigma\big((1-P(T,\Cpx\setminus\overline{A}))T\big),
\]
where the spectra are computed in the compressions of $\Mcal$ by $1-P(T,\Cpx\setminus A)$ and $1-P(T,\Cpx\setminus\overline{A})$, respectively.
Thus,~\eqref{eq:sig1-A} follows from~\ref{it:sigmaF}.

The implications \ref{it:sigmaA}$\implies$\ref{it:sigmaG} and \ref{it:sigmaG}$\implies$\ref{it:sigmaGbdy} follow immediately.

We now show \ref{it:sigmaGbdy}$\implies$\ref{it:Tdec}, by appealing to Theorem~\ref{thm:lw}.
Suppose $G$ and $H$ are open disks in $\Cpx$ with $\Gbar\subseteq H$.
There exists an open disk $D$ such that $G\subseteq D\subseteq\barD\subseteq H$ and $\nu_T(\partial D)=0$.
Let $Y=P(T,\Cpx\setminus D)\Hc$ and $Z=P(T,\barD)$.
Then
\begin{align*}
\text{eqn.~\eqref{eq:sigGc}}&\implies\sigma(TP_Y)\subseteq\Cpx\setminus D\subseteq\Cpx\setminus G, \\
\text{eqn.~\eqref{eq:sig1-G}}&\implies\sigma((1-P_Y)T)\subseteq\barD\subseteq H, \\
\text{eqn.~\eqref{eq:sigG}}&\implies\sigma(TP_Z)\subseteq\barD\subseteq H, \\
\text{eqn.~\eqref{eq:sig1-Gc}}&\implies\sigma((1-P_Z)T)\subseteq\Cpx\setminus D\subseteq\Cpx\setminus G,
\end{align*}
as required.
\end{proof}

\begin{lem}\label{lem:PBA}
Suppose $T\in\Mcal$ is decomposable.
Let $A$ and $B$ be Borel subsets of $\Cpx$ with $A\subseteq B$.
Then
\begin{equation}\label{eq:sigBA}
\sigma\big((P(T,B)-P(T,A))T(P(T,B)-P(T,A))\big)\subseteq \overline{B}\cap\overline{\Cpx\setminus A},
\end{equation}
where the spectrum is computed in the compression of $\Mcal$ by $P(T,B)-P(T,A)$.
\end{lem}
\begin{proof}
Let $q_1=P(T,B)$ and $q_2=P(T,A)$.
For a projection $r\in\Mcal$, let $\sigma^{(r)}(\cdot)$ denote the spectrum computed in $r\Mcal r$.

Since $1-q_2$ is $T$-coinvariant, we have
\[
(q_1-q_2)T(q_1-q_2)=(q_1-q_2)(q_1Tq_1),
\]
and $q_1-q_2$ is $q_1Tq_1$-coinvariant.
Thus, by Lemma~\ref{lem:utspec},
\[
\sigma^{(q_1-q_2)}((q_1-q_2)T(q_1-q_2))\subseteq\sigma^{(q_1)}(q_1Tq_1)\subseteq\overline{B},
\]
where the last inclusion is from Proposition~\ref{prop:disks}\ref{it:sigmaA}.

Furthermore, 
\[
(q_1-q_2)T(q_1-q_2)=(1-q_2)T(1-q_2)q_1,
\]
so $q_1$ is $(1-q_2)T(1-q_2)$-invariant.
Thus, by Lemma~\ref{lem:utspec},
\[
\sigma^{(q_1-q_2)}((q_1-q_2)T(q_1-q_2))\subseteq\sigma^{(1-q_2)}((1-q_2)T(1-q_2))\subseteq\overline{\Cpx\setminus A},
\]
Taking the intersection yields~\eqref{eq:sigBA}.
\end{proof}

A bounded operator $T$ on Hilbert space $\Hc$ is said to be {\em strongly decomposable} if, for every closed subset $F\subseteq\Cpx$, the restriction
of $T$ to the local spectral subspace $\Hc_T(F)$ is decomposable.
(See~\cite{LN00} for more.)
This entails, of course, that $T$ is itself decomposable.
Thus, if $T\in\Mcal$, then by Haagerup and Schultz's result, Proposition~\ref{prop:HSdecomp}, $T$ is strongly decomposable if and only if for every such $F$,
the operator $TP(T,F)$, considered as an element of $P(T,F)\Mcal P(T,F)$, is decomposable.
We get the following:
\begin{prop}\label{prop:strdec}
Let $T\in\Mcal$.
Then $T$ is strongly decomposable if and only if, whenever $E$ and $F$ are closed subsets of $\Cpx$,
we have
\begin{equation}\label{eq:PTFE}
\sigma\big((P(T,F)-P(T,F\setminus E))T(P(T,F)-P(T,F\setminus E))\big)\subseteq F\cap E,
\end{equation}
where the spectrum is computed in the compression of $\Mcal$ by the projection $P(T,F)-P(T,F\setminus E)$.
\end{prop}
\begin{proof}
$T$ is strongly decomposable if and only if for all closed subsets $K\subseteq\Cpx$, $TP(T,K)$ is decomposable.
Fix $K$ and write $S=TP(T,K)$ and $Q=P(T,K)$.
Using Proposition~\ref{prop:disks}\ref{it:sigmaF}, $S$ is decomposable if and only if, for all closed subsets $L\subseteq\Cpx$,
we have
\begin{align*}
\sigma(SP^{(Q)}(S,L))&\subseteq L \\ 
\sigma((Q-P^{(Q)}(S,\Cpx\setminus L))S)&\subseteq L, 
\end{align*}
where $P^{(Q)}(S,\cdot)$ is the Haagerup--Schultz projection computed in $Q\Mcal Q$ and where
the spectra are computed in the compressions of $\Mcal$ by the projections $P^{(Q)}(S,L)$ and $Q-P^{(Q)}(S,\Cpx\setminus L)$, respectively.
From Theorems~\ref{thm:Pcompress} and~\ref{thm:HSprojLattice}, we have 
\[
P^{(Q)}(S,L)=P(T,K\cap L),\qquad P^{(Q)}(S,\Cpx\setminus L)=P(T,K\setminus L).
\]
Thus, we must show that~\eqref{eq:PTFE} holds for all $E$ and $F$ if and only if
\begin{gather}
\sigma(TP(T,L\cap K))\subseteq L, \label{eq:TPLK} \\
\sigma\big((P(T,K)-P(T,K\setminus L))TP(T,K)\big)\subseteq L \label{eq:PKLT}
\end{gather}
hold for all $K$ and $L$.

Assume that~\eqref{eq:PTFE} holds for all $E$ and $F$.
Taking $E=F$, we get
\[
\sigma(TP(T,F))=\sigma(P(T,F)TP(T,F))\subseteq F.
\]
This implies~\eqref{eq:TPLK} holds for all $K$ and $L$.
On the other hand, since $1-P(T,K\setminus L)$ is $T$-coinvariant, we get
\begin{multline}
(P(T,K)-P(T,K\setminus L))TP(T,K) \\
=(P(T,K)-P(T,K\setminus L))T(P(T,K)-P(T,K\setminus L)), \label{eq:PKLTcoinv}
\end{multline}
so applying~\eqref{eq:PTFE} using $F=K$ and $E=L$, we find that~\eqref{eq:PKLT} holds for all $K$ and $L$.

Assume that~\eqref{eq:TPLK} and~\eqref{eq:PKLT} hold for all $K$ and $L$.
Let $E$ and $F$ be closed subsets of $\Cpx$.
Then using~\eqref{eq:PKLTcoinv}, from~\eqref{eq:PKLT} with $K=F$ and $L=E$ we get
\[
\sigma\big((P(T,F)-P(T,F\setminus E))T(P(T,F)-P(T,F\setminus E))\big)\subseteq E.
\]
On the other hand, from~\eqref{eq:TPLK} with $L=K=F$, we get
\[
\sigma(TP(T,F))\subseteq F.
\]
Since $1-P(T,F\setminus E)$ is $T$-coinvariant, using Lemma~\ref{lem:utspec}, we get
\[
\sigma\big((P(T,F)-P(T,F\setminus E))T(P(T,F)-P(T,F\setminus E))\big)\subseteq\sigma(P(T,F)TP(T,F))\subseteq F.
\]
Combining, we have~\eqref{eq:PTFE}.
\end{proof}

\begin{rem}
If we assume only that $T$ is decomposable, then from Lemma~\ref{lem:PBA}, instead of the inclusion~\eqref{eq:PTFE} we get
\begin{multline*}
\sigma\big((P(T,F)-P(T,F\setminus E))T(P(T,F)-P(T,F\setminus E))\big) \\
\subseteq F\cap\overline{\Cpx\setminus(F\setminus E)}
=(F\cap E)\cup(F\cap\overline{\Cpx\setminus F})
=(F\cap E)\cup\partial F.
\end{multline*}
\end{rem}

The remainder of this section is devoted to introducing the notion of Borel decomposability and exhibiting some examples.

\begin{defi}\label{def:fss}
Let $T\in\Mcal$.
We say that $T$ has {\em full spectral distribution} if $\supp(\nu_T)=\sigma(T)$.
\end{defi}

The following shows that there is no ambiguity when considering full spectral distribution of elements that can be taken in a corner $p\Mcal p$ 
of a von Neumann algebra.
\begin{lem}\label{prop:pTp}
Suppose $p$ is a nonzero projection in $\Mcal$ and $T=pTp\in\Mcal$.
If we use $\sigma^{(p)}(T)$ and $\nu_T^{(p)}$ to indicate the spectrum and, respectively, Brown measure computed in $p\Mcal p$
with respect to the trace $\tau(p)^{-1}\tau|_{p\Mcal p}$,
while using $\sigma(T)$ and $\nu_T$ to indicate the quantities computed in $\Mcal$ and with respect to $\tau$, then
\begin{equation}\label{eq:sigmanup}
\supp(\nu_T)=\sigma(T)\quad\Longleftrightarrow\quad\supp(\nu_T^{(p)})=\sigma^{(p)}(T).
\end{equation}
\end{lem}
\begin{proof}
We have
\[
\sigma(T)=\sigma^{(p)}(T)\cup\{0\},\qquad\supp(\nu_T)=\supp(\nu_T^{(p)})\cup\{0\}.
\]
Thus, the implication $\Longleftarrow$ always holds in~\eqref{eq:sigmanup}.
Suppose, for contradiction, that the implication $\Longrightarrow$ fails to hold.
Then we must have $\supp(\nu_T^{(p)})\subsetneq\sigma^{(p)}(T)$ and $\{0\}=\sigma^{(p)}(T)\setminus\supp(\nu_T^{(p)})$.
However, by Proposition~\ref{prop:concomp}, the point $0$ cannot be an isolated point of $\sigma^{(p)}(T)$.
Thus, there is a sequence in $\sigma^{(p)}(T)\setminus\{0\}$ that converges to $0$.
This same sequence lies in $\supp(\nu_T^{(p)})$, so $0\in\supp(\nu_T^{(p)})$, a contradiction.
\end{proof}

\begin{prop}\label{prop:totdiscfsd}
If $T\in\Mcal$ has totally disconnected spectrum, then $T$ has full spectral distribution.
\end{prop}
\begin{proof}
If $\sigma(T)$ is totally disconnected, then every singleton subset of $\sigma(T)$ is a connected component of $\sigma(T)$
so by Proposition~\ref{prop:concomp}, is a subset of $\supp(\nu_T)$.
\end{proof}

Other examples are the DT-operators, introduced in~\cite{DH04}.
These include Voi\-cu\-les\-cu's circular operator and the circular free Poisson operators.
The following is a consequence of Theorem~5.2 and Corollary~5.5 of~\cite{DH04}.
\begin{prop}\label{prop:DTfsd}
If $Z\in\Mcal$ is a DT-operator, then $Z$ has full spectral distribution.
\end{prop}

\begin{defi}\label{def:fssdecomp}
Let $T\in\Mcal$.
We will say that $T$ is {\em Borel decomposable}
if, for all Borel subsets $X$ and $Y$ of $\Cpx$ such that $X\subseteq Y$ and $\nu_T(X)<\nu_T(Y)$,
letting $p=P(T,Y)-P(T,X)$, the element $pTp$ has full spectral distribution.
\end{defi}

That the name ``Borel decomposable'' is appropriate can be seen from the characterization found in Proposition~\ref{prop:strdec}
for strong decomposability, 
compared to the result below.
See also Lemma~\ref{lem:PBA} for the case of decomposable operators.
\begin{prop}\label{prop:BoreldecAB}
Let $T\in\Mcal$.
Then $T$ is Borel decomposable if and only if for all Borel sets $A,B\subseteq\Cpx$
with $A\subseteq B$,
we have
\begin{equation}\label{eq:sigPTAB}
\sigma\big((P(T,B)-P(T,A))T(P(T,B)-P(T,A))\big)\subseteq\overline{B\setminus A},
\end{equation}
where the spectrum is computed in the compression of $\Mcal$ by $P(T,B)-P(T,A)$.
\end{prop}
\begin{proof}
Suppose $T$ is Borel decomposable.
If $\nu_T(B)=\nu_T(A)$, then~\eqref{eq:sigPTAB} holds because the indicated spectrum is, by convention, empty.
So assume $\nu_T(B)>\nu_T(A)$.
By the essential property of the Haagerup--Schultz projections (Theorem~\ref{thm:hsproj}) and the fact (Theorem~\ref{thm:Pcompress})
that $P(T,A)$
is also the Haagerup--Schultz projection of the operator $TP(T,B)$ for the set $A$ computed in $P(T,B)\Mcal P(T,B)$,
we have that the Brown measure of the operator 
\begin{equation}\label{eq:PTABop}
(P(T,B)-P(T,A))T(P(T,B)-P(T,A))
\end{equation}
computed in the compression of $\Mcal$ by $P(T,B)-P(T,A)$
equals the renormalized restriction of $\nu_T$ to $B\setminus A$.
Thus, the support of this Brown measure is contained in $\overline{B\setminus A}$.
Since $T$ is assumed to be Borel decomposable, the spectrum of the operator~\eqref{eq:PTABop} is equal to the support of its Brown measure, and~\eqref{eq:sigPTAB} holds.

On the other hand, suppose that the inclusion~\eqref{eq:sigPTAB} holds for every $A$ and $B$, and let us show that $T$ is Borel decomposable.
Let $X$ and $Y$ be Borel sets such that $X\subseteq Y$ and $\nu_T(X)<\nu_T(Y)$.
Let $p=P(T,Y)-P(T,X)$.
We must show that $pTp$ has full spectral distribution.
The same reasoning as above shows that the Brown measure of $pTp$ is the renormalized restriction of $\nu_T$ to the set $Y\setminus X$.
Let
\[
A=X\cap\supp(\nu_T),\qquad B=Y\cap\supp(\nu_T).
\]
Then
$P(T,A)=P(T,X)$ and $P(T,B)=P(T,Y)$.
So by~\eqref{eq:sigPTAB}, the spectrum of $pTp$ is contained in the set
\[
\overline{B\setminus A}=\overline{(Y\setminus X)\cap\supp(\nu_T)}.
\]
Thus, the spectrum of $pTp$ is contained in the support of its Brown measure.
Since the opposite inclusion holds for all operators, the two sets are equal.
Thus, $pTp$ has full spectral distribution.
\end{proof}

\begin{prop}\label{prop:totdiscBoreldec}
Suppose $T\in\Mcal$ has totally disconnected spectrum.
Then $T$ is Borel decomposable.
\end{prop}
\begin{proof}
Suppose $X$ and $Y$ are Borel subsets of $\Cpx$ such that $\nu_T(X)<\nu_T(Y)$.
Since $P(T,X)$ and $P(T,Y)$ are $T$-invariant projections, we may write $T$ as an upper triangular $3\times 3$ matrix with respect to the projections
$P(T,X)$, $p:=P(T,Y)-P(T,X)$ and $1-P(T,Y)$ and the operator $pTp$ appears as a diagonal entry.
Thus, by Lemma~\ref{lem:utspec}, the spectrum of $pTp$ is closed subset of $\sigma(T)$ and is, therefore, totally disconnected.
By Proposition~\ref{prop:totdiscfsd}, $pTp$ has full spectral distribution.
\end{proof}

\begin{prop}\label{prop:DTBoreldec}
If $Z\in\Mcal$ is a DT-operator, then $Z$ is Borel decomposable.
\end{prop}
\begin{proof}
In the notation of~\cite{DH04}, $Z$ is a $\DT(\nu_Z,c)$-operator for some $c>0$, where $\nu_Z$ is the Brown measure of $Z$
(see Corollary~5.5 of~\cite{DH04}).
By Theorem~5.8 of~\cite{DH04}, $Z$ is decomposable.
Suppose $X$ and $Y$ are Borel sets with $\nu_Z(X)<\nu_Z(Y)$ and let $p=P(Z,Y)-P(Z,X)$.
We must show that $pZp$ has full spectral distribution.
We will see that $pZp$ is itself a DT-operator with respect to the renormalized trace $\tau(p)^{-1}\tau|_{p\Mcal p}$,
which, by Proposition~\ref{prop:DTfsd}, suffices.
Let $Q=P(T,Y)$.
Using Theorem~5.4 of~\cite{DH04} and Proposition~\ref{prop:HSdecomp}, we have that $QZQ=ZQ$ is a DT-operator, with respect to the renormalized
trace $\tau(Q)^{-1}\tau|_{Q\Mcal Q}$.
But, by Theorem~\ref{thm:Pcompress}, 
$P^{(Q)}(ZQ,X)=P(T,X)$, so again applying Theorem~5.4 of~\cite{DH04} and Proposition~\ref{prop:HSdecomp}, we have that
$pZp=(Q-P^{(Q)}(ZQ,X))Z(Q-P^{(Q)}(ZQ,X))$ is a DT-operator, as required.
\end{proof}

\section{Main theorem and related results}
\label{sec:mainthm}

We now turn to notation and preliminary results for the proof of our main theorem (Theorem~\ref{thm:decqn}).
When $\psi$ is a continuous spectral ordering for $T\in\Mcal$, we will use the notation for $0\le a<b\le1$:
\begin{align*}
R^\psi_{[a,b]}&=P(T,\psi([0,b]))-P(T,\psi([0,a))), \\
R^\psi_{(a,b]}&=P(T,\psi([0,b]))-P(T,\psi([0,a])), \\
R^\psi_{(a,b)}&=P(T,\psi([0,b)))-P(T,\psi([0,a])), \\
R^\psi_{[a,b)}&=P(T,\psi([0,b)))-P(T,\psi([0,a))).
\end{align*}
We may write simply $R_{[a,b]}$ instead of $R^\psi_{[a,b]}$, {\em etc.}, if the choice of $\psi$ is clear from context.
Moreover, for any interval $I$ in $\Reals$, we will let $R_I=R_{I\cap[0,1]}$.

\begin{lem}\label{lem:QNspec}
Let $T\in\Mcal$.
If $T=N+Q$ is an upper-triangular form arising from a continuous spectral ordering $\psi$ and if $Q$ is quasinilpotent,
then $T$ has full spectral distribution.
\end{lem}
\begin{proof}
In any upper triangular form we have $\sigma(N)=\supp(\nu_T)\subseteq\sigma(T)$.
Let $b\in\sigma(T)$.
Let
\[
a=\inf\{t\in[0,1]\mid b\in\sigma(R_{[0,t]}TR_{[0,t]})\},
\]
where the spectrum is computed in $R_{[0,t]}\Mcal R_{[0,t]}$.
Note that, since $R_{[0,t]}$ is a $T$-invariant projection, by Lemma~\ref{lem:utspec},
the set $\sigma(R_{[0,t]}TR_{[0,t]})$ increases as $t$ increases.
We claim that
\begin{equation}\label{eq:bsig}
\forall\eps>0\qquad b\in\sigma(R_{(a-\eps,a+\eps)}TR_{(a-\eps,a+\eps)}).
\end{equation}
Indeed, by choice of $a$ and using Lemma~\ref{lem:utspec}, we have
\[
b\in\sigma(R_{[0,a+\frac\eps2]}TR_{[0,a+\frac\eps2]})\subseteq\sigma(R_{[0,a+\eps)}TR_{[0,a+\eps)}),
\qquad
b\notin\sigma(R_{[0,a-\eps]}TR_{[0,a-\eps]})
\]
Thus, by applying Lemma~\ref{lem:utspec} again, the claim is proved.
In particular, we have $R_{(a-\eps,a+\eps)}\ne0$.
Since $R_{(a-\eps,a+\eps)}$ is a subprojection of the spectral projection of $N$ for the set $\psi((a-\eps,a+\eps)\cap[0,1])$
and since $\psi$ is continuous,
we have that $\sigma(N)$ meets
\[
\bigcap_{\eps>0}\psi((a-\eps,a+\eps)\cap[0,1])=\{\psi(a)\}.
\]
Namely, we have $\psi(a)\in\sigma(N)=\supp(\nu_T)$.

We will show $\psi(a)=b$, which will complete the proof.
Suppose, for contradiction, $\psi(a)\ne b$.
Then, since $Q$ is quasinilpotent, the operator $\psi(a)-b+Q$ is invertible.
Let $\delta>0$.
Let $\eps>0$ be such that
\begin{equation}\label{eq:psia+}
\psi([0,1]\cap(a-\eps,a+\eps))\subseteq\psi(a)+\frac{\delta}{2}\mathbb{D}.
\end{equation}
By~\eqref{eq:bsig}, there exists
a unit vector $x\in R_{(a-\eps,a+\eps)}\mathcal{H}$ such that 
\begin{equation}\label{eq:RTx}
\Vert R_{(a-\eps,a+\eps)}TR_{(a-\eps,a+\eps)}x-bx\Vert_\mathcal{H} < \frac{\delta}{2}.
\end{equation}
Since $R_{(a-\eps,a+\eps)}$ is a subprojection of the spectral projection of $N$
for the set $\psi([0,1]\cap(a-\eps,a+\eps))$, by
using~\eqref{eq:psia+}, we see 
\[
\|\psi(a)R_{(a-\eps,a+\eps)}-R_{(a-\eps,a+\eps)}N\|<\frac\delta 2
\]
and, by writing $T=N+Q$, also
\begin{equation}\label{eq:nmdelta2}
\Vert R_{(a-\eps,a+\eps)}TR_{(a-\eps,a+\eps)}-\psi(a)R_{(a-\eps,a+\eps)}-R_{(a-\eps,a+\eps)}QR_{(a-\eps,a+\eps)}\Vert < \frac{\delta}{2}.
\end{equation}
Combining \eqref{eq:RTx} and~\eqref{eq:nmdelta2}, we get 
\begin{equation}\label{eq:QRxdelta}
\Vert(\psi(a)-b+R_{(a-\eps,a+\eps)}QR_{(a-\eps,a+\eps)})x\Vert_{\mathcal{H}}<\delta.
\end{equation}
We may write $Q$ as a $3\times 3$ upper triangular matrix with respect to the projections $R_{[0,a-\eps]}$, $R_{(a-\eps,a+\eps)}$ and $R_{[a+\eps,1]}$ and 
likewise for $(\psi(a)-b+Q)^{-1}$.
The middle diagonal entry of this matrix is the inverse of $(\psi(a)-b)R_{(a-\eps,a+\eps)}+R_{(a-\eps,a+\eps)}QR_{(a-\eps,a+\eps)}$ on the Hilbert space
$R_{(a-\eps,a+\eps)}\mathcal{H}$.
However, by~\eqref{eq:QRxdelta}, this inverse has norm at least $\delta^{-1}$.
Thus, we get
\[
\|(\psi(a)-b+Q)^{-1}\|\ge\delta^{-1}.
\]
Since $\delta$ is arbitrary, this yields a contradiction.
\end{proof}

\begin{lem}\label{lem:disksNCP}
Let $T\in\Mcal$.
 Suppose that, writing $P_s$ for $P(T,s\overline{\mathbb{D}})$,
\begin{enumerate}[label=(\roman*),leftmargin=25pt]
\item\label{suff1} for every $s\geq 0$ such that $P_s\ne0$, the spectrum of $TP_s$ is contained in $s\overline{\mathbb{D}}$,
\item\label{suff2} for every $s>0$ such that $P_s\ne1$, the spectrum of $(1-P_s)T$, considered as an element of $(1-P_s)\Mcal (1-P_s)$,
is contained in the complement of $s\mathbb{D}$.
\end{enumerate}
Then $T$ has the norm convergence property.
\end{lem}
\begin{proof} Fix $\varepsilon>0$ small and choose $m\in\mathbb{N}$ such that $\varepsilon m\in(1,2)$.
We let $T=N+Q$ be an upper triangular form of $T$ with respect to some spectral ordering.
We may without loss of generality assume $\|N\|\leq 1$.
We have
\begin{align*}
\|T^n(|N|+\varepsilon)^{-n}\|&\leq\sum_{k=0}^{m-1}\left\|T^n(|N|+\varepsilon)^{-n}\chi_{[\frac{k}{m},\frac{k+1}{m})}(|N|)\right\| \\
&\leq\sum_{k=0}^{m-1}\left(\frac{k}{m}+\varepsilon\right)^{-n}\left\|T^n\chi_{[\frac{k}{m},\frac{k+1}{m})}(|N|)\right\| \\
&\leq\sum_{k=0}^{m-1}\left(\frac{k}{m}+\varepsilon\right)^{-n}\left\|T^n\chi_{[0,\frac{k+1}{m}]}(|N|)\right\| \\
&=\sum_{k=0}^{m-1}\left(\frac{k}{m}+\varepsilon\right)^{-n}\left\|\left(TP_{\frac{k+1}m}\right)^n\right\|.
\end{align*}
By hypothesis~\ref{suff1} for large enough $n$, we have
$$\left\|T^n(|N|+\varepsilon)^{-n}\right\|\leq\sum_{k=0}^{m-1}\left(\frac{k}{m}+\varepsilon\right)^{-n}(1+\varepsilon)^n\left(\frac{k+1}{m}\right)^n
\leq m(1+\varepsilon)^n\leq\frac{2}{\varepsilon}(1+\varepsilon)^n,$$
where we have used for the second inequality $\frac1m<\varepsilon$ and for the last inequality $m<\frac2\varepsilon$.
This implies
$$(|N|+\varepsilon)^{-n}(T^*)^nT^n(|N|+\varepsilon)^{-n}\leq \frac{4}{\varepsilon^2}(1+\varepsilon)^{2n},$$
$$|T^n|^2\leq \frac{4}{\varepsilon^2}(1+\varepsilon)^{2n}(|N|+\varepsilon)^{2n}.$$
Since the mapping $t\to t^{\frac1{2n}}$ is operator monotone, it follows that
\begin{equation}\label{suff estimate from above}
|T^n|^{\frac1n}\leq (1+\varepsilon)\left(\frac{2}{\varepsilon}\right)^{\frac1n}(|N|+\varepsilon).
\end{equation}

Let $\delta>0$.
By hypothesis~\ref{suff2}, $(1-P_\delta)T$ is invertible in $(1-P_\delta)\Mcal(1-P_\delta)$.
Let $R_\delta$ denote its inverse.
We have
\begin{align}
\|(|N|-\varepsilon)_+^nR_\delta^n\|
&\leq\sum_{k=0}^{m-1}\left\|\chi_{(\frac{k}{m},\frac{k+1}{m}]}(|N|)(|N|-\varepsilon)_+^nR_\delta^n\right\| \label{eq:NR1} \displaybreak[1]\\
&\leq\sum_{k=0}^{m-1}\left(\frac{k+1}{m}-\varepsilon\right)_+^n\left\|\chi_{(\frac{k}{m},\infty)}(|N|)R_\delta^n\right\| \label{eq:NR2} \displaybreak[2]\\
&\leq\sum_{k=1}^{m-1}\left(\frac{k}{m}\right)^n\left\|(1-P_{\frac km})R_\delta^n\right\| \label{eq:NR3} \displaybreak[1]\\
&=\sum_{k=1}^{m-1}\left(\frac{k}{m}\right)^n\left\|((1-P_{\frac km})R_\delta)^n\right\|, \label{eq:NR4}
\end{align}
where for~\eqref{eq:NR1} we used $\chi_{\{0\}}(|N|)(N-\varepsilon)_+=0$ and for~\eqref{eq:NR3} we used $\frac1m<\varepsilon$.
Assume $\delta<\frac1m$, which ensures $P_\delta\le P_{\frac km}$ for all $k\ge1$.
Since the projection $(1-P_\delta)P_{\frac km}$ is $(1-P_\delta)T(1-P_\delta)$-invariant,
by Lemma~5(i) of~\cite{DSZ.a} we have,
\[
(1-P_{\frac km})R_\delta=(1-P_{\frac km})R_\delta(1-P_{\frac km}),
\]
which we used to obtain the equality~\eqref{eq:NR4},
and we also have
\begin{align*}
(1-P_{\frac km})R_\delta&=\left((1-P_{\frac km})(1-P_\delta)T(1-P_\delta)(1-P_{\frac km})\right)^{-1} \\
&=\left((1-P_{\frac km})T(1-P_{\frac km})\right)^{-1},
\end{align*}
where the inverse is taken in $(1-P_{\frac km})\Mcal(1-P_{\frac km})$.
In particular, $(1-P_{\frac km})R_\delta$ is independent of $\delta$, so long as $\delta<\frac1m$.
By hypothesis~\ref{suff2}, 
the spectrum of $(1-P_{\frac km})T$ lies in the complement of $\frac km\mathbb{D}$,
so the spectrum of $(1-P_{\frac km})R_\delta$ lies $\frac mk\overline{\mathbb{D}}$.
Thus, continuing from \eqref{eq:NR1}--\eqref{eq:NR4}, for large enough $n$ (independent of $\delta$), we get
\[
\left\|(|N|-\varepsilon)_+^nR_\delta^n\right\|
\leq\sum_{k=1}^{m-1}\left(\frac{k}{m}\right)^n(1+\varepsilon)^n\left(\frac{m}{k}\right)^n
=(m-1)(1+\varepsilon)^n\leq\frac{2}{\varepsilon}(1+\varepsilon)^n.
\]
Therefore, we have (for all sufficiently large $n$)
\[
(R_\delta^*)^n(|N|-\varepsilon)_+^{2n}R_\delta^n\leq \frac{4}{\varepsilon^2}(1+\varepsilon)^{2n}.
\]
Multiplying on the right by $R_\delta^{-n}$ and on the left by its adjoint, we get
\begin{align*}
(1-P_\delta)(|N|-\varepsilon)_+^{2n}&\leq\frac{4}{\varepsilon^2}(1+\varepsilon)^{2n}((1-P_\delta)T^*(1-P_\delta))^n((1-P_\delta)T(1-P_\delta))^n \\
&=\frac{4}{\varepsilon^2}(1+\varepsilon)^{2n}(1-P_\delta)(T^*)^n(1-P_\delta)T^n(1-P_\delta) \\
&\le\frac{4}{\varepsilon^2}(1+\varepsilon)^{2n}(1-P_\delta)(T^*)^nT^n(1-P_\delta).
\end{align*}
Letting $\delta\to0$, by strong operator convergence, we get
$$
(|N|-\varepsilon)_+^{2n}=
(1-P_0)(|N|-\varepsilon)_+^{2n}\leq\frac{4}{\varepsilon^2}(1+\varepsilon)^{2n}(1-P_0)|T^n|^2(1-P_0).$$
Clearly,
$$(1-P_0)|T^n|^2(1-P_0)=|T^n|^2+(-(T^*)^n(TP_0)^n-(P_0T^*)^nT^n+(P_0T^*)^n(TP_0)^n).$$
By hypothesis~\ref{suff1}, $TP_0$ is quasinilpotent. Hence, for sufficiently large $n$ we have
$$\|-(T^*)^n(TP_0)^n-(P_0T^*)^nT^n+(P_0T^*)^n(TP_0)^n\|\leq\varepsilon^{2n}.$$
Thus,
$$(|N|-\varepsilon)_+^{2n}\leq\frac{4}{\varepsilon^2}(1+\varepsilon)^{2n}(|T^n|^2+\varepsilon^{2n})\leq\frac{4}{\varepsilon^2}(1+\varepsilon)^{2n}(|T^n|^{\frac1n}+\varepsilon)^{2n}.$$
Since the mapping $t\to t^{\frac1{2n}}$ is operator monotone, we get
$$(|N|-\varepsilon)_+\leq (1+\varepsilon)\left(\frac{2}{\varepsilon}\right)^{\frac1n}(|T^n|^{\frac1n}+\varepsilon).$$
It follows that
\begin{equation}\label{suff estimate from below}
|N|\leq (1+\varepsilon)\left(\frac{2}{\varepsilon}\right)^{\frac1n}(|T^n|^{\frac1n}+2\varepsilon).
\end{equation}

Combining \eqref{suff estimate from above} and \eqref{suff estimate from below}, we have, for all $n$ sufficiently large,
\[
\left(\frac\varepsilon2\right)^{\frac1n}(1+\varepsilon)^{-1}|N|-2\varepsilon\leq
|T^n|^{\frac1n}\leq (1+\varepsilon)\left(\frac{2}{\varepsilon}\right)^{\frac1n}(|N|+\varepsilon).
\]
Since $\varepsilon>0$ is arbitrarily small, $T$ has the norm convergence property.
\end{proof}

\begin{thm}\label{thm:decqn}
Let $T\in\Mcal$.
Consider the following properties:
\begin{enumerate}[label=(\alph*),leftmargin=20pt]
\item\label{it:fssdec} $T$ is Borel decomposable,
\item\label{it:strdec} $T$ is strongly decomposable,
\item\label{it:Qqn} for every continuous spectral ordering of $T$, in the corresponding upper triangular form
$T=N+Q$, the operator $Q$ is quasinilpotent,
\item\label{it:dec} $T$ is decomposable,
\item\label{it:sncp} $T$ has the shifted norm convergence property,
\item\label{it:fss} $T$ has full spectral distribution.
\end{enumerate}
Then the implications \ref{it:fssdec}$\implies$\ref{it:strdec}$\implies$\ref{it:Qqn}$\implies$\ref{it:dec}$\implies$\ref{it:sncp}$\implies$\ref{it:fss}
hold.
\end{thm}
\begin{proof}
The implication \ref{it:fssdec}$\implies$\ref{it:strdec} follows immediately, comparing Propositions~\ref{prop:strdec} and~\ref{prop:BoreldecAB}.

\smallskip
We now show \ref{it:strdec}$\implies$\ref{it:Qqn}.
Suppose $T$ is strongly decomposable, $\psi$ is a continuous spectral ordering for $T$ and $T=N+Q$ is the corresponding upper triangular form of $T$.
Suppose, for contradiction, that $Q$ is not quasinilpotent.
We may without loss of generality assume $1\in\sigma(Q)$.
Let
\[
a=\inf\{t\in[0,1]\mid 1\in\sigma(R_{[0,t]}QR_{[0,t]})\},
\]
where the spectrum is computed in $R_{[0,t]}\Mcal R_{[0,t]}$.
Note that, since $(R_{[0,t]})_{0\le t\le 1}$ is an increasing family of $Q$-invariant projections, by Lemma~\ref{lem:utspec},
the set $\sigma(R_{[0,t]}QR_{[0,t]})$ increases as $t$ increases.
We claim that
\begin{equation}\label{eq:1sig}
\forall\eps>0\qquad 1\in\sigma(R_{[a-\eps,a+\eps]}QR_{[a-\eps,a+\eps]}).
\end{equation}
Indeed, by choice of $a$, we have
\[
1\in\sigma(R_{[0,a+\eps]}QR_{[0,a+\eps]}),
\qquad
1\notin\sigma(R_{[0,a-\eps]}QR_{[0,a-\eps]})\supseteq\sigma(R_{[0,a-\eps)}QR_{[0,a-\eps)}).
\]
Thus, by applying Lemma~\ref{lem:utspec} again, the claim is proved.
In particular, we have $R_{[a-\eps,a+\eps]}\ne0$.

For notational convenience, given an interval $I$, we will understand $\psi(I)$ to mean $\psi(I\cap[0,1])$.
Since $R_{[a-\eps,a+\eps]}$ is a subprojection of the spectral projection of $N$ for the set $\psi([a-\eps,a+\eps])$ and since $\psi$ is continuous,
we have
\begin{equation}\label{eq:psiaNto0}
\lim_{\eps\searrow0}\|\psi(a)R_{[a-\eps,a+\eps]}-R_{[a-\eps,a+\eps]}N\|=0.
\end{equation}
Using~\eqref{eq:1sig} we have
\begin{equation}\label{eq:psia1sig}
\psi(a)+1\in\sigma(\psi(a)R_{[a-\eps,a+\eps]}+R_{[a-\eps,a+\eps]}QR_{[a-\eps,a+\eps]}).
\end{equation}
Since $\psi$ is continuous, we can choose $\eps_0>0$ such that
\[
\psi([a-\eps_0,a+\eps_0])\subseteq \psi(a)+\frac12\mathbb{D}.
\]
We have
\[
\psi([0,a+\eps_0])\setminus\psi([a-\eps_0,a+\eps_0])\subseteq\psi([0,a-\eps_0))\subseteq\psi([0,a+\eps_0]).
\]
Let
\[
\Rt=P\big(T,\psi([0,a+\eps_0])\big)-P\big(T,\psi([0,a+\eps_0])\setminus\psi([a-\eps_0,a+\eps_0])\big).
\]
Then
\[
R_{[a-\eps_0,a+\eps_0]}=P\big(T,\psi([0,a+\eps_0])\big)-P\big(T,\psi([0,a-\eps_0))\big)\le\Rt
\]
and 
\[
R_{[a-\eps_0,a+\eps_0]}=\big(1-P(T,\psi([0,a-\eps_0)))\big)\Rt.
\]
Note that the projections $\Rt$ and $1-P(T,\psi([0,a-\eps_0)))$ commute.
Since $1-P(T,\psi([0,a-\eps_0)))$ is $T$-coinvariant, using Lemma~\ref{lem:utspec}, we have
\[
\sigma(R_{[a-\eps_0,a+\eps_0]}TR_{[a-\eps_0,a+\eps_0]})\subseteq\sigma(\Rt T\Rt).
\]
Since $\psi([0,a+\eps_0])$ and $\psi([a-\eps_0,a+\eps_0])$ are closed in $\Cpx$, by
the hypothesis that $T$ is strongly decomposable and Proposition~\ref{prop:strdec}, we have
\begin{equation}\label{eq:siginball}
\sigma(R_{[a-\eps_0,a+\eps_0]}TR_{[a-\eps_0,a+\eps_0]})\subseteq\sigma(\Rt T\Rt)\subseteq\psi([a-\eps_0,a+\eps_0])\subseteq\psi(a)+\frac12\mathbb{D}.
\end{equation}

For all $\eps<\eps_0$, $R_{[a-\eps_0,a+\eps_0]}TR_{[a-\eps_0,a+\eps_0]}$
is upper triangular with respect to the projections $R_{[a-\eps_0,a-\eps)}$, $R_{[a-\eps,a+\eps]}$ and $R_{(a+\eps,a+\eps_0]}$
and the middle entry is
\[
R_{[a-\eps,a+\eps]}N+R_{[a-\eps,a+\eps]}QR_{[a-\eps,a+\eps]}.
\]
When we replace this middle entry with $\psi(a)R_{[a-\eps,a+\eps]}+R_{[a-\eps,a+\eps]}QR_{[a-\eps,a+\eps]}$,
using~\eqref{eq:psia1sig} and Lemma~\ref{lem:utspec}, we get
\[
\psi(a)+1\in\sigma(\psi(a)R_{[a-\eps,a+\eps]}-R_{[a-\eps,a+\eps]}N+R_{[a-\eps_0,a+\eps_0]}TR_{[a-\eps_0,a+\eps_0]}).
\]
Using~\eqref{eq:psiaNto0} and the fact that the set of invertible elements is open, by letting $\eps\to0$ we conclude
\[
\psi(a)+1\in\sigma(R_{[a-\eps_0,a+\eps_0]}TR_{[a-\eps_0,a+\eps_0]}).
\]
Now using~\eqref{eq:siginball}, we get
\[
\psi(a)+1\in\sigma(R_{[a-\eps_0,a+\eps_0]}TR_{[a-\eps_0,a+\eps_0]}\subseteq\psi(a)+\frac12\overline{\mathbb{D}},
\]
a contradiction.

\smallskip
We now prove \ref{it:Qqn}$\implies$\ref{it:dec}.
Assuming \ref{it:Qqn}, we will show that $T$ is decomposable by verifying condition~\ref{it:sigmaGbdy} of Proposition~\ref{prop:disks}.
Given an open disk $D\subseteq\Cpx$ such that $\nu_T(\partial D)=0$, we will verify that the inclusions \eqref{eq:sigG}--\eqref{eq:sig1-Gc} hold.
Choose a continous spectral ordering $\psi$ for $T$ such that
\[
\psi([0,1/2])=\barD,\qquad\psi([1/2,1])\subseteq\Cpx\setminus D
\]
and let $T=N+Q$ be the corresponding upper triangular form of $T$.
By hypothesis, $Q$ is quasinilpotent.
Recall how $N$ is constructed in~\cite{DSZ}:  it is the image of $T$ under the
conditional expectation onto the von Neumann algebra generated by the set $\{P(T,\psi([0,r]))\mid 0\le r\le 1\}$
of commuting projections.
Thus, $P(T,D)=P(T,\barD)$ is the spectral projection for $N$ of $D$ and of $\barD$.
Since the projection $P(T,\barD)$ commutes with $N$ and is invariant under $T$, it is also invariant under $Q$, and, thus,
by Corollary~\ref{cor:utquasi},
both $QP(T,\barD)$ and $(1-P(T,D))Q$ are quasinilpotent.

We will now show that the inclusion~\eqref{eq:sigG} holds.
Firstly, note that if $P(T,\barD)=0$, then there is nothing to prove, so we may assume $P(T,\barD)\ne0$.
We claim that
\begin{equation}\label{eq:TPut}
TP(T,\barD)=NP(T,\barD)+QP(T,\barD)
\end{equation}
is the upper triangular form of $S:=TP(T,\barD)$ corresponding to the continuous spectral ordering arising from the restriction of the function $\psi$ to $[0,1/2]$,
where $S$ is treated as an element of the compressed algebra $P(T,\barD)\Mcal P(T,\barD)$.
Note that $\nu_S$ is the renormalized restriction of $\nu_T$ to $\barD$.
Using Theorem~\ref{thm:Pcompress} and the lattice properties of Theorem~\ref{thm:HSprojLattice}, we have $P(S,X)=P(T,X\cap\barD)$
for every Borel set $X\subseteq\Cpx$, where $P(S,X)$ is computed in the compression of $\Mcal$.
In particular, we have $P(S,\psi([0,t]))=P(T,\psi([0,t]))$ for every $t\in[0,1/2]$ and we see that $NP(T,\barD)$ is the image of $S$ under the trace preserving
conditional
expectation from $P(T,\barD)\Mcal P(T,\barD)$ onto the von Neumann algebra generated by $\{P(S,\psi([0,r])\mid 0\le r\le 1/2\}$, and the claim
about~\eqref{eq:TPut} is proved.
Thus, since $QP(T,\barD)$ is quasinilpotent, using Lemma~\ref{lem:QNspec}, we conclude
\[
\sigma(TP(T,\barD))=\sigma(NP(T,\barD))\subseteq\barD,
\]
and~\eqref{eq:sigG} is proved.

We will now prove that~\eqref{eq:sig1-G} holds.
If $P(T,\barD)=1$, then there is nothing to prove, so we may assume $P(T,\barD)\ne1$.
We claim that
\begin{equation}\label{eq:1-PTut}
(1-P(T,\barD))T=(1-P(T,\barD))N+(1-P(T,\barD))Q
\end{equation}
is the upper triangular form of $R:=(1-P(T,\barD))T$ corresponding to the continuous
spectral ordering arising from the restriction of the function $\psi$ to $[1/2,1]$,
where $R$ is treated as an element of the compressed algebra $(1-P(T,\barD))\Mcal(1-P(T,\barD))$.
Note that $\nu_R$ is the renormalized restriction of $\nu_T$ to $\Cpx\setminus D$.
Using Theorem~\ref{thm:Pcompress} and the lattice properties of Theorem~\ref{thm:HSprojLattice}, we have
\[
P(R,X)=P(T,\barD)\vee P(T,X)-P(T,\barD)=P(T,\barD\cup X)-P(T,\barD)
\]
for every Borel set $X\subseteq\Cpx$, where $P(R,X)$ is computed in the aforementioned compression of $\Mcal$.
In particular, we have $P(R,\psi([1/2,t]))=P(T,\psi([0,t]))-P(T,\barD)$ for every $t\in[1/2,1]$ and we see that $(1-P(T,\barD))N$ is the image of $R$ under the trace preserving
conditional
expectation from $(1-P(T,\barD))\Mcal(1-P(T,\barD))$ onto the von Neumann algebra generated by
\[
\{P(R,\psi([1/2,t]))-P(T,\barD)\mid 1/2\le t\le1\}.
\]
Thus, the claim about~\eqref{eq:1-PTut} is proved.
Since $(1-P(T,\barD))Q$ is quasinilpotent, using Lemma~\ref{lem:QNspec}, we conclude
\[
\sigma((1-P(T,\barD))T)=\sigma((1-P(T,\barD))N)\subseteq\Cpx\setminus D,
\]
and~\eqref{eq:sig1-G} is proved.

The inclusions~\eqref{eq:sigGc} and~\eqref{eq:sig1-Gc} are proved exactly analogously to the above proofs of~\eqref{eq:sigG} and~\eqref{eq:sig1-G},
but starting with a continuous spectral ordering $\psi$ for $T$, that satisfies
\[
\supp(\nu_T)\cap(\Cpx\setminus D)\subseteq\psi([0,1/2])\subseteq\Cpx\setminus D,\qquad\psi([1/2,1])=\barD.
\]
This completes the proof that $T$ is decomposable.

\smallskip
We now prove \ref{it:dec}$\implies$\ref{it:sncp}.
Suppose $T\in\Mcal$ is decomposable.
Since decomposability does not change if we add a scalar multiple of the identity to an operator, it will suffice to show that $T$ has
the norm convergence property.
This follows easily from Lemma~\ref{lem:disksNCP}.
Indeed, condition~\ref{suff1} of that lemma is immediate from decomposability of $T$, and condition~\ref{suff2} follows by the following familiar argument.
The projection $1-P(T,s\overline{\mathbb{D}})$ is a $T$-coinvariant subprojetion of $1-P(T,s\mathbb{D})$, so the spectrum of
$(1-P(T,s\overline{\mathbb{D}}))T$ is a subset of the spectrum of $(1-P(T,s\mathbb{D}))T$, and the latter is contained in $\Cpx\setminus s\mathbb{D}$,
by decomposability of $T$, using condition~\ref{it:sigmaF} of Proposition~\ref{prop:disks}.

\smallskip
We now prove \ref{it:sncp}$\implies$\ref{it:fss}.
We suppose $T\in\Mcal$ has the shifted norm convergence property
and must show $\supp(\nu_T)=\sigma(T)$.
Of course, the inclusion $\subseteq$ holds generally.
To show the reverse inclusion, suppose $\zeta\in\Cpx\setminus\supp(\nu_T)$.
Then $0\notin\supp(\nu_{T-\zeta})$ and, consequently, by Theorem~\ref{thm:hsconv}, letting $A$ be the strong operator limit as $n\to\infty$
of $|(T-\zeta)^n|^{1/n}$, $A$ is invertible.
By hypothesis, this convergence to $A$ holds also in norm topology.
Consequently, for sufficiently large $n$, $|(T-\zeta)^n|^{1/n}$ is invertible and, thus, $T-\zeta$ is invertible.
So $\zeta\notin\sigma(\nu_T)$.
\end{proof}

The proof of the following (essentially, just constructing an example) is more conveniently postponed until after Proposition~\ref{prop:Qeq}, below.
However, it is natural to state it here.
\begin{prop}\label{prop:nofimpliesa}
Consider the conditions listed in Theorem~\ref{thm:decqn}.
The implication \ref{it:fss}$\implies$\ref{it:sncp} fails to hold.
\end{prop}

The following questions are natural:
\begin{ques2}
Let $T\in\Mcal$.
\begin{enumerate}[label=(\roman*),leftmargin=25pt]
\item\label{it:impl} Do any other implications among the conditions \ref{it:fssdec}--\ref{it:sncp} of Theorem~\ref{thm:decqn} hold?
In particular, are the properties of being Borel decomposable, strongly decomposable and decomposable distinct for elements of
finite von Neumann algebras?
(The construction by Eschmeier~\cite{E88} of a decomposable Hilbert space operator that is not strongly decomposable is of interest.)
\item\label{it:quessncp} If $T$ has the shifted norm convergence property, must also $T^*$ have it?
(See Theorem~\ref{thm:totdisc}.)
\item\label{it:quesoneQ} If there exists a continuous spectral ordering for $T$ such that, in the corresponding Schur-type upper triangular form $T=N+Q$,
the operator $Q$ is quasinilpotent, must any of the conditions \ref{it:fssdec}--\ref{it:fss} of Theorem~\ref{thm:decqn} hold?
(See Theorem~\ref{thm:fsequivs}.)
\end{enumerate}
\end{ques2}

\begin{thm}\label{thm:totdisc}
Suppose $T\in\Mcal$ and $\supp(\nu_T)$ is totally disconnected.
Then all the conditions \ref{it:fssdec}--\ref{it:fss} of Theorem~\ref{thm:decqn} are equivalent to each other and are also equivalent to each of the following conditions:
\begin{enumerate}[label=(\alph*),leftmargin=20pt]
\setcounter{enumi}{6}
\item\label{it:totdiscspec} $\sigma(T)$ is totally disconnected,
\item\label{it:T*} $T^*$ has the shifted norm convergence property.
\end{enumerate}
\end{thm}
\begin{proof}
The implication \ref{it:totdiscspec}$\implies$\ref{it:fssdec} follows immediately from Proposition~\ref{prop:totdiscBoreldec},
while the implication \ref{it:fss}$\implies$\ref{it:totdiscspec} is clear.
Thus, by Theorem~\ref{thm:decqn}, conditions \ref{it:fssdec}--\ref{it:totdiscspec} are equivalent when $\supp(\nu_T)$ is totally disconnected.
It is clear, however, from the properties of Brown measure and of spectrum in finite von Neumann algebras,
that for any element of $S\in\Mcal$, $S$ has full spectral distribution if and only if $S^*$ has full spectral distribution.
This implies that condition~\ref{it:T*} is also equivalent to \ref{it:fssdec}--\ref{it:totdiscspec}.
\end{proof}

\section{An s.o.t.-quasinilpotent operator with thin spectrum}
\label{sec:thinspec}

Our main purpose in this section is to construct an s.o.t.-quasinilpotent operator with spectrum that is 
equal to a nondegenerate interval in the real line.

Let $f\in L^\infty(\Tcirc)$ and let $(a_k)_{k\in\Ints}$ be its Fourier coefficients:
\[
f(e^{i\theta})\sim\sum_{k\in\Ints}a_ke^{ik\theta}.
\]
Let $M_f$ denote the operator of multiplication by $f$ on $L^2(\Tcirc)$.
Consider the usual orthonormal basis $(v_k)_{k\in\Ints}$ for $L^2(\Tcirc)$ where $v_k(e^{i\theta})=e^{ik\theta}$.
Writing $M_f$ with respect to this orthonormal basis, we have the Laurent operator $L(f):=(a_{\ell-k})_{k,\ell\in\Ints}$.
Let $P$ be the projection of $L^2(\Tcirc)$ onto $\clspan\{v_k\mid k\ge0\}$ and for $n\ge0$, let $P_n$ be the
projection of $L^2(\Tcirc)$ onto $\lspan\{v_k\mid 0\le k\le n\}$.
Then $T(f):=PL(f)P$ is the Toeplitz operator $T(f)=(a_{\ell-k})_{k,\ell\ge0}$ and $T_n(f):=P_nL(f)P_n$ is the Toeplitz matrix
$T_n(f)=(a_{\ell-k})_{0\le k,\ell\le n}$.

The next result contains well known facts about norms of Toeplitz matrices;
for convenience we give a brief proof of them.
For more refined results in the self-adjoint case, see~\cite{HW54}.
\begin{prop}\label{prop:MfLT}
\begin{enumerate}[label=(\alph*)]
\item\label{it:nmL} $\|L(f)\|=\|f\|_\infty$.
\item\label{it:nmTnT} For every $n\ge0$, $\|T_n(f)\|\le\|T_{n+1}(f)\|\le\|T(f)\|$.
\item\label{it:nmTnlim} $\|L(f)\|=\|T(f)\|=\lim_{n\to\infty}\|T_n(f)\|$.
\end{enumerate}
\end{prop}
\begin{proof}
Parts~\ref{it:nmL} and~\ref{it:nmTnT} are immediate from the definitions, as is the inequality $\|T(f)\|\le\|L(f)\|$.
Let $\eps>0$.
There exist $x,y\in\lspan\{v_k\mid k\in\Ints\}$ such that $\|x\|=\|y\|=1$ and $|\langle L(f)x,y\rangle|\ge\|L(f)\|-\eps$.
Since $L(f)$ commutes with the bilateral shift operator, we may without loss of generality assume
$x,y\in\lspan\{v_k\mid 0\le k\le n\}$ for some $n\ge0$.
Thus,
\[
\|T_n(f)\|\ge|\langle L(f)x,y\rangle|\ge\|L(f)\|-\eps.
\]
This implies~\ref{it:nmTnlim}.
\end{proof}

The mappings $f\mapsto L(f)$, $f\mapsto T(f)$ and $f\mapsto T_n(f)$ are, of course, linear and $*$-preserving, so we have
\[
T_n(\Re f)=\Re T_n(f),\qquad T_n(\Im f)=\Im T_n(f).
\]

\begin{thm}\label{thm:thinspec}
In the hyperfinite II$_1$-factor, there exists an s.o.t.-quasinilpotent operator
whose spectrum is a nondegenerate interval in the real line.
\end{thm}
\begin{proof}
Let $f_n$ be a conformal mapping from the unit disk onto 
\[
\left\{a+ib\,\bigg|\,-1<a<1,\,-\frac1n<b<\frac1n\right\}.
\]
that satifies $f_n(0)=0$.
Then, of course, for all $n$, we have $\|\Re f_n\|_\infty=1$ and $\|\Im f_n\|_\infty=\frac1n$, so $\lim_{n\to\infty}\|f_n\|_\infty=1$.
Since $f_n$ is holomorphic and $f_n(0)=0$, its Fourier coefficients $a_k$ vanish for $k\le0$.
Thus, the Toeplitz matrix $T_k(f_n)$ is strictly upper triangular and, hence, nilpotent, for each $k\ge1$.
Applying Proposition~\ref{prop:MfLT}, we get a sequence $(k(n))_{n=1}^\infty$ of positive integers such that
\[
\lim_{n\to\infty}\|T_{k(n)}(f_n)\|=1=\lim_{n\to\infty}\|\Re T_{k(n)}(f_n)\|.
\]
By a standard construction, we can realize the von Neumann algebra direct sum
\[
\Mcal=\bigoplus_{n=1}^\infty M_{k(n)}(\Cpx)
\]
as a von Neumann subalgebra of the hyperfinite II$_1$-factor.
Let
\[
Q=\oplus_{n=1}^\infty T_{k(n)}(f_n)\in\Mcal.
\]
Let $A_n=\Re(T_{k(n)}(f_n))$ and $B_n=\Im(T_{k(n)}(f_n))$.

\begin{claim}\label{cl:sotqn}
\noindent
$Q$ is s.o.t.-quasinilpotent.
\end{claim}
Since for each $n$, $T_{k(n)}(f_n)$ is nilpotent and since in $\Mcal$,
the projection $0^{\oplus n}\oplus 1\oplus 1\oplus\cdots$ converges in strong operator topology to $0$ as $n\to\infty$,
this is clear.

\begin{claim}\label{cl:sr}
\noindent
The spectral radius of $Q$ is at least $1$.
\end{claim}
For each $m\in\Nats$, we have $\|Q^m\|\ge\limsup_{n\to\infty}\|T_{k(n)}^m\|$
and for each $n$, we have
\[
\|T_{k(n)}^m\|\ge\|A_n^m\|-\sum_{k=1}^m\binom mk\|A_n\|^{m-k}\|B_n\|^k.
\]
Since $\|A_n^m\|=\|A_n\|^m$ and $\lim_{n\to\infty}\|A_n\|=1$, while $\lim_{n\to\infty}\|B_n\|=0$, we get $\|Q^m\|\ge1$.
This proves Claim~\ref{cl:sr}.

\begin{claim}\label{cl:realspec}
The spectrum of $Q$ lies in $\Reals$.
\end{claim}
\noindent
Suppose $\lambda\in\Cpx\setminus\Reals$.
We have $\lambda-Q=\bigoplus_{n=1}^\infty(\lambda-T_{k(n)}(f_n))$.
Since for each $n$, $T_{k(n)}(f_n)$ is nilpotent, each $\lambda-T_{k(n)}(f_n)$ is invertible.
To show $\lambda-Q$ is invertible, it will suffice to show
\[
\sup_{n\ge1}\|(\lambda-T_{k(n)}(f_n))^{-1}\|<\infty.
\]
As soon as $n>2/|\Im\lambda|$, we have $\|B_n\|\le|\Im\lambda|/2$ and, thus, $\|(\Im\lambda-B_n)^{-1}\|\le2/|\Im\lambda|$.
We also have
\begin{multline*}
(\lambda-T_{k(n)}(f_n))^{-1}
=\big((\Im\lambda-B_n)i+(\Re\lambda-A_n)\big)^{-1} \\
=|\Im\lambda-B_n|^{-1/2}\big(\pm i+|\Im\lambda-B_n|^{-1/2}(\Re\lambda-A_n)|\Im\lambda-B_n|^{-1/2}\big)^{-1}|\Im\lambda-B_n|^{-1/2},
\end{multline*}
where the sign in $\pm i$ is the sign of $\Im\lambda$.
Since the operator
\[
|\Im\lambda-B_n|^{-1/2}(\Re\lambda-A_n)|\Im\lambda-B_n|^{-1/2}
\]
is self-adjoint, the operator
\[
\pm i+|\Im\lambda-B_n|^{-1/2}(\Re\lambda-A_n)|\Im\lambda-B_n|^{-1/2}
\]
is normal and has inverse of norm $\le1$, so we get
\[
\left\|(\lambda-T_{k(n)}(f_n))^{-1}\right\|\le\frac2{|\Im\lambda|}.
\]
This shows that $\lambda-Q$ is invertible,
and Claim~\ref{cl:realspec} is proved.

From Claims~\ref{cl:sotqn}, \ref{cl:sr} and \ref{cl:realspec} and the fact that the spectrum of an s.o.t.-quasinilpotent
operator must be connected and contain the point $0$, it follows that $Q$ is an s.o.t.-quasinilpotent
operator in the hyperfinite II$_1$-factor
whose spectrum is an interval in $\Reals$ containing $0$ and at least one of the points $\pm1$.
\end{proof}

\section{Norm convergence properties of s.o.t.-quasinilpotent operators}
\label{sec:ncpsotqn}

We begin by proving two lemmas for quasinilpotent operators.

\begin{lem}\label{lem:qnsum}
Suppose an element $Q$ of a C$^*$-algebra
is quasinilpotent.
Then the series $\sum_{k=1}^\infty\|Q^k\|$ converges and, hence, the series
$\sum_{k=1}^\infty Q^k$
converges in norm to an operator that is quasinilpotent.
\end{lem}
\begin{proof}
For any $\delta>0$, there exists $N_0 \in \mathbb{N}$ such that for all $k \geq N_0$, we have $\Vert Q^k \Vert \leq \delta^k$.
Thus, we see that the series $\sum_{k=1}^\infty\|Q^k\|$ converges.
From this, we get that the series $\sum_{k=0}^\infty Q^k$ converges in norm 
to a bounded operator $R$ that commutes with $Q$, and the series $\sum_{k=1}^\infty Q^k$
converges in norm to the bounded operator $RQ$.
Now standard estimates show that $RQ$ is quasinilpotent.
\end{proof}

\begin{lem}\label{lem:qninv}
Suppose an element $Q$ of a unital C$^*$-algebra
is quasinilpotent.
Then $1+Q$ is invertible and
$(1+Q)^{-1} = 1+S$,
where $S$ is quasinilpotent.
\end{lem}

\begin{proof}
By Lemma \ref{lem:qnsum}, the series 
$\sum_{k=1}^\infty (-1)^kQ^k$
converges to a quasinilpotent operator $S$.  We easily see that $1+S=(1+Q)^{-1}$.
\end{proof}

\begin{lem}\label{lem:qnconv}
Suppose an element $Q$ of a unital C$^*$-algebra
is quasinilpotent.  Then
$$\lim_{n \to \infty} \big\Vert \vert (1+Q)^n\vert ^{1/n}-1 \big\Vert =0.$$
Thus, $1+Q$ has the norm convergence property.
\end{lem}
\begin{proof}
We must show that for all $\varepsilon>0$ there exists $N_0 \in \mathbb{N}$ such that, for all $n \geq N_0$, we have 
$$(1-\varepsilon)^{2n} \leq (1+Q^*)^n(1+Q)^n \leq (1+\varepsilon)^{2n}.$$
Since $1+Q$ is invertible, it will suffice to show that for all $n \geq N_0$ we have
\begin{align}
\Vert (1+Q)^n \Vert&\leq (1+\varepsilon)^n\label{(1+q)^n} \\
\Vert (1+Q)^{-n} \Vert &\leq (1-\varepsilon)^{-n}\label{(1+q)^-n}.
\end{align}

To show \eqref{(1+q)^n}, let $N_1$ be such that $\Vert Q \Vert^n \leq \left(\frac{\varepsilon}{2} \right)^n$ for every $n\geq N_1$.  Using the binomial formula, for $n \geq N_1$ we have 
\begin{align*}
\Vert (1+Q)^n \Vert &\leq \sum_{k=0}^n \binom nk\Vert Q^k\Vert
\leq \sum_{k=0}^{N_1-1}\binom nk\Vert Q^k \Vert + \sum_{k=N_1}^n\binom nk\left( \frac{\varepsilon}{2} \right)^k\\
&\leq \left(1+\frac{\varepsilon}{2}\right)^n + \sum_{k=0}^{N_1-1}n^k\Vert Q \Vert ^k
\leq \left(1+\frac{\varepsilon}{2}\right)^n + N_1(1+n\Vert Q \Vert)^{N_1}.
\end{align*}
Since 
$$\lim_{n \to \infty}\frac{\log\big((1+\frac{\varepsilon}{2})^n + N_1(1+n\Vert Q \Vert\big)^{N_1})}{n} = \log\left(1+\frac{\varepsilon}{2}\right)<\log(1+\varepsilon),$$
we get that \eqref{(1+q)^n} holds for $n$ large enough.

By Lemma \ref{lem:qninv}, $(1+Q)^{-1}=1+S$ for some quasinilpotent $S$.  Hence applying \eqref{(1+q)^n} in the case of this operator $S$ implies that \eqref{(1+q)^-n} holds for $n$ large enough.
\end{proof}

\begin{prop}\label{prop:Qeq}
Suppose $Q\in B(\mathcal{H})$ is s.o.t.-quasinilpotent.
Then the following are equivalent:
\begin{enumerate}[label=(\alph*)]
\item\label{it:QeqQqn} $Q$ is quasinilpotent,
\item\label{it:Qncp} $Q$ has the norm convergence property,
\item\label{it:Qsncp} $Q$ has the shifted norm convergence property.
\end{enumerate}
\end{prop}
\begin{proof}
The equivalence of~\ref{it:QeqQqn} and~\ref{it:Qncp} is clear.
The implication \ref{it:Qsncp}$\implies$\ref{it:Qncp} is also immediate,
while the implication \ref{it:QeqQqn}$\implies$\ref{it:Qsncp} follows from Lemma~\ref{lem:qnconv}.
Indeed, assuming~\ref{it:QeqQqn}, for every nonzero $\lambda\in\Cpx$, $\lambda^{-1}Q$ is also quasinilpotent.
So after rescaling, from Lemma~\ref{lem:qnconv} we have that $\lambda+Q$ has the norm convergence property.
Of course, $Q$ itself has the norm convergence property, and~\ref{it:Qsncp} is proved.
\end{proof}

\begin{proof}[Proof of Proposition~\ref{prop:nofimpliesa}]
The proof consists of constructing an example of $T\in\Mcal$ that has full spectral distribution but does not satisfy the shifted norm convergence property.
Let $T\in\Mcal$ be a direct sum $T=N\oplus Q\in\Mcal_1\oplus\Mcal_2$, where $N$ is a normal operator whose Brown measure
(taken in $\Mcal_1$) is supported on the closed unit disk and has no atoms, and where $Q$ is any s.o.t.-quasinilpotent operator 
that is not quasinilpotent but whose spectrum is contained in the closed unit disk;
for example, take the operator described in the introduction around equation~\eqref{eq:QJordan}.
Then both the spectrum of $T$ and the support of its Brown measure are the closed unit disk, so $T$ has full spectral distribution.
However, $T$ has the norm convergence property if and only if both $N$ and $Q$ have it.
But $Q$ does not have the norm convergence property, by Proposition~\ref{prop:Qeq}.
\end{proof}

We will show, using the following proposition in Example~\ref{ex} below, that the equivalence of~\ref{it:Qncp} and~\ref{it:Qsncp}
in Proposition~\ref{prop:Qeq} does not extend even to operators having Brown measure supported at
exactly one (nonzero) point.

\begin{prop}\label{prop:1+Qncp}
Let $Q\in B(\mathcal{H})$ be s.o.t.-quasinilpotent.
Then $1+Q$ has the norm convergence property if and only if $\sigma(1+Q) \subseteq \mathbb{T}$.
\end{prop}
\begin{proof}
The proof is a straightforward application of the spectral radius formula.  Suppose first that $\sigma(1+Q) \subseteq \mathbb{T}$.  Since the spectral radii of $1+Q$ and $(1+Q)^{-1}$ are both $1$, for any $\varepsilon>0$, there exists $N\in\Nats$ such that for all $n>N$, $\Vert (1+Q)^n\Vert<(1+\varepsilon)^n$ and $\Vert (1+Q)^{-n}\Vert<(1+\varepsilon)^n$.  Thus we have 
$$(1+Q^*)^n(1+Q)^n\leq \Vert (1+Q)^n\Vert^2\leq (1+\varepsilon)^{2n}$$
and
$$(1+Q^*)^n(1+Q)^n\geq (1+Q^*)^n\frac{(1+Q^*)^{-n}(1+Q)^{-n}}{\Vert (1+Q^*)^{-n}(1+Q)^{-n}\Vert}(1+Q)^n\geq (1+\varepsilon)^{-2n}.$$
Hence,
$$(1+\varepsilon)^{-1}\leq \vert (1+Q)^n\vert^{1/n}\leq 1+\varepsilon.$$
It follows that the sequence $\vert (1+Q)^n\vert^{1/n}$ converges in norm to $1$.

Now assume that $1+Q$ has the norm convergence property.  Since $\vert (1+Q)^n\vert^{1/n}$ must converge in norm to $1$, it follows that for sufficiently large $n$, $\vert (1+Q)^n\vert^{1/n}$ is invertible, so that also $(1+Q)^n$ and $1+Q$ are invertible.  We observe now that for large $n$ we have 
$$(1-\varepsilon)^{2n}\leq (1+Q^*)^n(1+Q)^n\leq (1+\varepsilon)^{2n}$$
so that 
$$(1+\varepsilon)^{-2n}\leq (1+Q)^{-n}(1+Q^*)^{-n}\leq (1-\varepsilon)^{-2n}.$$
Hence, $(1+Q^*)^{-1}$ has the norm convergence property.

In addition, since $\vert (1+Q)^n\vert^{1/n}$ must converge in norm to $1$, for large $n$ we have $\Vert (1+Q)^n\Vert < (1+\varepsilon)^n$,
so by the spectral radius formula, $\sigma(1+Q)\subseteq \overline{\mathbb{D}}$.
Applying the same argument to $(1+Q^*)^{-1}$ we see that $\sigma(1+Q^*)\subseteq \mathbb{C}\setminus\mathbb{D}$.  Thus, 
$$\sigma(1+Q)\subseteq \overline{\mathbb{D}} \bigcap \left(\mathbb{C}\setminus\mathbb{D}\right) = \mathbb{T},$$
as desired.
\end{proof}

Here is the promised example, that serves both to show that Proposition~\ref{prop:Qeq} does not extend, and to show that the naive guess
at an answer to Question~\ref{ques:basic} is wrong.
\begin{ex}\label{ex}
Let $Q$ be the s.o.t.-quasinilpotent element of the hyperfinite II$_1$--factor that was constructed in Section~\ref{sec:thinspec},
with $\sigma(Q)$ a nondegenerate interval in the real line.
Using either the holomorphic functional calculus and the main result of~\cite{DSZ.a}, or arguing more directly with power series,
we have that
the operator $\exp(iQ)$ is of the form $1+S$, where $S$ is s.o.t.-quasinilpotent and $\exp(iQ)$ has spectrum contained in $\mathbb{T}$
and consists of more than just the point $\{1\}$.
Hence, by Proposition~\ref{prop:1+Qncp}, $\exp(iQ)$ has the norm convergence property.
However, $\sigma(\exp(iQ)-1)\neq \{0\}$, so $S=\exp(iQ)-1$ is not quasinilpotent and
does not have the norm convergence property.
Therefore, $\exp(iQ)$ does not have the shifted norm convergence property.
\end{ex}

\section{Operators with finitely supported Brown measure}
\label{sec:finsupp}

In this section, we focus on elements $T\in\Mcal$ whose Brown measures have finite support.

If the Brown measure of the operator $T$ has more than one point in its support, then it is possible to construct distinct upper-triangular forms of $T$.  Theorem \ref{thm:hsconv} tells us that the strong operator limit, $A$, of the sequence $|T^n|^{1/n}$ has as spectral projections $P(T,r\overline{\mathbb{D}})$ for $r\in[0,\Vert T\Vert]$.
Thus, if the Brown measure of $T$ has support
$a_1,\ldots,a_m$
ordered so that $|a_1|\leq |a_2|\leq\cdots\leq|a_m|$, then the corresponding upper-triangular form $T=N+Q$ is upper-triangular with respect to the spectral projections of $A$ for disks centered at the origin.  Theorem \ref{thm:hsconv} thus implies that $A=|N|$.  We use this fact to show that if the upper-triangular part $Q$ for such an ordering is quasinilpotent, then $T$ has the norm convergence property.

\begin{lem}\label{lem:finqnthenncp}
Let $Q \in \mathcal{M}$ be quasinilpotent.
Let $(P_i)_{1 \leq i \leq m}$ be projections such that $\sum_{i=1}^m P_i =1$ and such that, for every $1 \leq k \leq m$,
the projection $\sum_{i=1}^k P_i$ is $Q$-invariant.
Let $a_1,\ldots,a_m$ be complex numbers with $0 \leq |a_1| \leq |a_2| \leq \cdots \leq |a_m|$.  Let $T = \sum_{i=1}^m a_iP_i + Q$.  Then $|T^n|^{1/n}$ converges in norm as $n\to\infty$ to $\sum_{i=1}^m |a_i|P_i$.  In particular, $T$ has the norm convergence property.
\end{lem}
\begin{proof}
If $a_m=0$, the result is clear, so assume $a_m \neq 0$.
The proof proceeds by induction.
The case $m=1$ follows from Lemma~\ref{lem:qnconv}.
For $m>1$, suppose the result holds for all $1\leq k<m$.

Let 
\[
N=\sum_{i=1}^{m-1} a_iP_i,\quad
Q_{11}=Q\sum_{i=1}^{m-1}P_i,\quad
Q_{12}=\sum_{i=1}^{m-1}P_iQP_m,\quad
Q_{22}=P_mQ.
\]
Note that by Lemma~\ref{lem:utspec}, $Q_{11}$ and $Q_{22}$ are quasinilpotent.

We may now write $T$ as the 2$\times$2 matrix, 
$$T = \left( \begin{matrix}
N+Q_{11} & Q_{12}\\
0 & a_m+Q_{22}
\end{matrix}
\right).$$
We also write
\begin{alignat*}{3}
A &= N+Q_{11}, &\quad
C &= a_m +Q_{22}, &\quad
B_n &= \sum_{k=0}^{n-1}A^kQ_{12}C^{-k-1},\\
R_n&=\left( \begin{matrix}
1 & B_n\\
0 & 1
\end{matrix}
\right),&\quad
G&=\left( \begin{matrix}
A & 0\\
0 & C
\end{matrix}
\right),
\end{alignat*}
so that $T^n = R_nG^n.$  Note that $R_n$ is invertible in $\Mcal$ and $C$ is invertible in $P_m\Mcal P_m$, with 
$$C^{-1}=a_m^{-1}+S$$
for a quasinilpotent operator $S$ (by Lemma \ref{lem:qninv}), and 
$$R_n^{-1}=\left( \begin{matrix}
1 & -B_n\\
0 & 1
\end{matrix}
\right).$$

Assume for the moment that $N\neq 0$.
Fix $\varepsilon\in(0,1)$.
By the inductive hypothesis, there exists $K_0 \in \mathbb{N}$ such that for any $k\geq K_0$, we have
\begin{gather}
\||A^k|^{\frac1k}\|\leq|a_{m-1}|(1+\varepsilon), \label{eq:Akbd} \\
\||C^k|^{\frac1k}\|\in(|a_m|(1-\varepsilon),|a_m|(1+\varepsilon)), \notag \\
\||C^{-k}|^{\frac1k}\|\in(|a_m|^{-1}(1-\varepsilon),|a_m|^{-1}(1+\varepsilon)). \notag
\end{gather}
Hence for $k > K_0$,
\begin{multline*}
\Vert A^kQ_{12}C^{-k-1}\Vert\leq\Vert Q_{12}C^{-1}\Vert\cdot\Vert A^k\Vert\cdot\Vert C^{-k}\Vert \\
\leq \Vert Q_{12}C^{-1}\Vert\cdot(|a_{m-1}|(1+\varepsilon))^k\cdot(|a_m|^{-1}(1+\varepsilon))^k
\leq \Vert Q_{12}C^{-1}\Vert\cdot(1+3\varepsilon)^k.
\end{multline*}
For $n>K_0,$ we have
\begin{multline*}
\Vert B_n \Vert = \left\Vert \sum_{k=0}^{K_0}A^kQ_{12}C^{-k-1} + \sum_{K_0+1}^{n-1} A^kQ_{12}C^{-k-1} \right\Vert \\
\le \left\Vert \sum_{k=0}^{K_0}A^kQ_{12}C^{-k-1} \right\Vert + n\Vert Q_{12}C^{-1}\Vert (1+3\varepsilon)^n.
\end{multline*}
Thus, for sufficiently large $n,$ we have
$\|R_n^{-1}\|=\|R_n\|\leq(1+4\varepsilon)^n$.
It follows that
\begin{align*}
|T^n|^2=(G^n)^*(R_n^*R_n)G^n&\leq\|R_n\|^2(G^n)^*G^n\leq(1+4\varepsilon)^{2n}(G^n)^*G^n, \\
|T^n|^2=(G^n)^*(R_n^*R_n)G^n&\geq\|R_n^{-1}\|^{-2}(G^n)^*G^n\geq(1+4\varepsilon)^{-2n}(G^n)^*G^n.
\end{align*}
Since the function $t\mapsto t^{1/2n}$ is operator monotone for positive operators, we get
$$(1+4\varepsilon)^{-1}|G^n|^{\frac1n}\leq|T^n|^{\frac1n}\leq(1+4\varepsilon)|G^n|^{\frac1n}$$
for all sufficiently large $n$.
By the inductive assumption, we have
$$|N|-\varepsilon\leq|G^n|^{\frac1n}\leq|N|+\varepsilon$$
for all sufficiently large $n.$ A combination of these two formulae gives us
$$(1+4\varepsilon)^{-1}(|N|-\varepsilon)\leq|T^n|^{\frac1n}\leq(1+4\varepsilon)(|N|+\varepsilon)$$
for all sufficiently large $n$.
Since $\varepsilon$ is arbitrary, it follows that
$|T^n|^{\frac1n}$ converges in norm to $|N|$ as $n\to\infty$, as desired.

If $N=0$, then the argument proceeds similarly, but where the estimate~\eqref{eq:Akbd} is replaced by $\||A^k|^{1/k}\|\le |a_m|/2$.
\end{proof}

We now show that for an operator with finitely supported Brown measure, the spectral
ordering used to construct an upper-triangular forms $T=N+Q$ 
does not affect whether s.o.t.-quasinilpotent part $Q$ is quasinilpotent.

\begin{lem}\label{lem:finoneqnthenallqn}
Suppose $T \in \mathcal{M}$ has finitely supported Brown measure.
Suppose that there exists an upper-triangular form $T=N+Q$ such that $Q$ is quasi\-nil\-pot\-ent.
If $T=\hat{N}+\hat{Q}$ is another upper-triangular form, then also $\hat{Q}$ is quasinilpotent.
\end{lem}
\begin{proof}
We have $N = \sum_{k=1}^n a_kP_k$,
for distinct complex numbers $a_1,\ldots,a_n$ and for projections $P_1,\ldots,P_n$ whose sum is $1$,
where
for $m \leq n$, $\sum_{k=1}^m P_k$ is the Haagerup-Schultz projection of $T$ associated with the set $\{ a_1,a_2,\ldots,a_m \} $.
Lemma \ref{lem:utspec} and Corollary \ref{cor:utquasi} imply that $\sigma(T)=\{ a_1,a_2,\ldots,a_n \}$.

Let $\{b_1,b_2,\ldots,b_n\}$ be any reordering of $\{ a_1,a_2,\ldots,a_n \}$,
and let $T = \hat{N}+\hat{Q}$ be the corresponding upper triangular form of $T$.
Then we have $\hat{N} = \sum_{k=1}^n b_kR_k$ and with $\sum_{k=1}^m R_k$ the Haagerup-Schultz projection of $T$ associated with the set
$\{ b_1,b_2,\ldots,b_m \} $ for all $m \leq n$.
Then Lemma \ref{lem:utspec} implies that $\sigma(R_kTR_k)$ contains only finitely many points for each $k$.
Combining this with Proposition~\ref{prop:concomp}, we see that $\supp(\nu_{R_kTR_k}) = \sigma(R_kTR_k)$ for each $k$.
Since $R_k\hat{Q}R_k$ is s.o.t.-quasinilpotent, we have for every $k$,
$$\sigma(R_kTR_k) = \supp(\nu_{R_kTR_k}) = \{ b_k\}.$$
As $R_kTR_k = b_k+R_k\hat{Q}R_k$, this implies that $R_k\hat{Q}R_k$ is quasinilpotent.  Since this is true for all $k$, Corollary \ref{cor:utquasi} implies that $\hat{Q}$ is quasinilpotent, completing the proof.
\end{proof}

\begin{thm}\label{thm:fsequivs}
Suppose $T\in\Mcal$ and $\supp(\nu_T)$ is finite.
Then all the conditions \ref{it:fssdec}--\ref{it:T*} of Theorems~\ref{thm:decqn} and~\ref{thm:totdisc}
are equivalent to each other and are also equivalent to each of the following conditions:
\begin{enumerate}[label=(\alph*),leftmargin=20pt]
\setcounter{enumi}{8}
\item\label{it:finitespec} $\sigma(T)$ is finite,
\item\label{it:existsqn} there exists a spectral ordering for $T$ for which, in the corresponding upper triangular form $T=N+Q$,
the s.o.t.-quasinilpotent operator $Q$ is quasinilpotent.
\end{enumerate}
\end{thm}
\begin{proof}
The equivalence of conditions \ref{it:fssdec}--\ref{it:finitespec} is clear from Theorems~\ref{thm:decqn} and~\ref{thm:totdisc}.
The equivalence of condition~\ref{it:existsqn} with~\ref{it:Qqn} follows from Lemma~\ref{lem:finoneqnthenallqn}.
\end{proof}


\begin{thebibliography}{19}

\bib{Br86}{article}{
  author={Brown, Lawrence G.},
  title={Lidskii's theorem in the type II case},
  conference={
    title={Geometric methods in operator algebras},
    address={Kyoto},
    date={1983},
  },
  book={
    series={Pitman Res. Notes Math. Ser.}, 
    volume={123},
    publisher={Longman Sci. Tech.},
    address={Harlow},
    date={1986},
  },
  pages={1--35},
}

\bib{CDSZ}{article}{
   author={Charlesworth, Ian},
   author={Dykema, Ken},
   author={Sukochev, Fedor},
   author={Zanin, Dmitriy},
   title={Joint spectral distributions and invariant subspaces for commuting operators in a finite von Neumann algebra},
   eprint={arXiv:1703.05695},
}

\bib{DH04}{article}{
   author={Dykema, Ken},
   author={Haagerup, Uffe},
   title={DT-operators and decomposability of Voiculescu's circular operator},
   journal={Amer. J. Math.},
   volume={126},
   date={2004},
   pages={121--189},
}
		
\bib{DS09}{article}{
  author={Dykema, Ken},
  author={Schultz, Hanne},
  title={Brown measure and iterates of the Aluthge transform for some operators arising from measureable actions},
  journal={Trans. Amer. Math. Soc.},
  volume={361},
  year={2009},
  pages={6583--6593}
}

\bib{DSZ}{article}{
   author={Dykema, Ken},
   author={Sukochev, Fedor},
   author={Zanin, Dmitriy},
   title={A decomposition theorem in $\rm II_1$-factors},
   journal={J. Reine Angew. Math.},
   volume={708},
   date={2015},
   pages={97--114},
}

\bib{DSZ.a}{article}{
   author={Dykema, Ken},
   author={Sukochev, Fedor},
   author={Zanin, Dmitriy},
   title={Holomorphic functional calculus on upper triangular forms in finite von Neumann algebras,},
   journal={Illinois J. Math.},
   volume={59},
   date={2015},
   pages={819--824},
}

\bib{E88}{article}{
   author={Eschmeier, J{\"o}rg},
   title={A decomposable Hilbert space operator which is not strongly decomposable},
   journal={Integral Equations Operator Theory},
   volume={11},
   date={1988},
   pages={161--172},
}

\bib{F63}{article}{
   author={Foia{\c{s}}, Ciprian},
   title={Spectral maximal spaces and decomposable operators in Banach space},
   journal={Arch. Math. (Basel)},
   volume={14},
   date={1963},
   pages={341--349},
}

\bib{Fr76}{article}{
   author={Frunz\u a, \c Stefan},
   title={Spectral decomposition and duality},
   journal={Illinois J. Math.},
   volume={20},
   date={1976},
   pages={314--321},
}

\bib{HS07}{article}{
   author={Haagerup, Uffe},
   author={Schultz, Hanne},
   title={Brown measures of unbounded operators affiliated with a finite von Neumann algebra},
   journal={Math. Scand.},
   volume={100},
   date={2007},
   pages={209--263},
}

\bib{HS09}{article}{
  author={Haagerup, Uffe},
  author={Schultz, Hanne},
  title={Invariant subspaces for operators in a general II$_1$--factor},
  journal={Publ. Math. Inst. Hautes \'Etudes Sci.},
  volume={109},
  year={2009},
  pages={19-111},
}

\bib{HW54}{article}{
   author={Hartman, Philip},
   author={Wintner, Aurel},
   title={The spectra of Toeplitz's matrices},
   journal={Amer. J. Math.},
   volume={76},
   date={1954},
   pages={867--882},
}

\bib{LW87}{article}{
   author={Lange, Ridgley},
   author={Wang, Sheng Wang},
   title={New criteria for a decomposable operator},
   journal={Illinois J. Math.},
   volume={31},
   date={1987},
   pages={438--445},
}

\bib{LN00}{book}{
   author={Laursen, Kjeld B.},
   author={Neumann, Michael M.},
   title={An introduction to local spectral theory},
   series={London Mathematical Society Monographs. New Series},
   volume={20},
   publisher={The Clarendon Press, Oxford University Press, New York},
   date={2000},
}

\bib{Nol}{article}{
   author={Noles, Joseph},
   title={Upper Triangular Forms and Spectral Orderings in a II$_1$-factor},
   status={preprint},
   eprint={http://arxiv.org/abs/1406.2774}
}

\end{thebibliography}
\end{document}